\documentclass[a4paper,11pt]{article}
\usepackage{stmaryrd}
\usepackage{amsfonts}
\usepackage{bbm}

\usepackage[all,cmtip]{xy}

\usepackage{amscd}
\usepackage{mathrsfs}
\usepackage{latexsym,amssymb,amsmath,amscd,amscd,amsthm,amsxtra}
\usepackage[dvips]{graphicx}
\usepackage[utf8]{inputenc}
\usepackage[T1]{fontenc}
\usepackage{enumerate}
\usepackage{lmodern}
\usepackage{amssymb}
\usepackage{nicefrac,mathtools,enumitem}
\usepackage{microtype}
\usepackage{amsfonts}
\usepackage{amssymb}
\usepackage{mathrsfs}
\usepackage{amsmath}
\usepackage{amsthm}
\usepackage{enumerate}
\usepackage{indentfirst}
\usepackage{amsfonts}
\usepackage{amssymb}
\usepackage{mathrsfs}
\usepackage{amsmath}
\usepackage{amsthm}
\usepackage{enumerate}
\usepackage{cite}
\usepackage{mathrsfs}
\usepackage{kpfonts}
\usepackage{geometry}
\usepackage{ifpdf}
\ifpdf
  \usepackage[colorlinks=true,linkcolor=blue,citecolor=red, final,backref=page,hyperindex]{hyperref}
\else
  \usepackage[colorlinks,final,backref=page,hyperindex,hypertex]{hyperref}
\fi

\allowdisplaybreaks

\makeatletter
\newcommand{\subjclass}[2][2020]{%
  \let\@oldtitle\@title%
  \gdef\@title{\@oldtitle\footnotetext{#1 \emph{Mathematics subject classification}: #2}}%
}
\newcommand{\keywords}[1]{%
  \let\@@oldtitle\@title%
  \gdef\@title{\@@oldtitle\footnotetext{\emph{Keywords}: #1}}%
}
\makeatother

\def\c{\cdot}
\def\r{\rho}

\newtheorem{thm}{Theorem}[section]
\newtheorem{prop}{Proposition}[section]

\newtheorem{rem}{Remark}[section]

\newtheorem{cor}{Corollary}[section]
\newtheorem{df}{Definition}[section]

\newtheorem{ex}{Example}[section]

\numberwithin{equation}{section}

\date{}
\begin{document}

\title{Malcev  Yang-Baxter equation, weighted   $\mathcal{O}$-operators  on  Malcev algebras  and post-Malcev algebras}
\author{F. Harrathi $^{1}$\footnote{E-mail: harrathifattoum285@gmail.com}\ ,
 S. Mabrouk $^{2}$\footnote{E-mail: mabrouksami00@yahoo.fr \text{(Corresponding author)}}\ ,
O. Ncib $^{2}$\footnote{E-mail: othmenncib@yahoo.fr}\ ,
S. Silvestrov $^{3}$\footnote{ E-mail: sergei.silvestrov@mdu.se }
}

\date{
$^{1}${\small University of Sfax, Faculty of Sciences Sfax,  BP 1171, 3038 Sfax, Tunisia} \\
$^{2}${\small  University of Gafsa, Faculty of Sciences Gafsa, 2112 Gafsa, Tunisia}\\
$^{3}${\small M\"{a}lardalen University,
Division of Mathematics and Physics,} \\
{\small
School of Education, Culture and Communication,} \\
\nopagebreak
{\small \hspace{0.5 cm} Box 883, 72123 V\"{a}steras, Sweden}
}

\maketitle

\begin{abstract}
The purpose of this paper is to study  the  $\mathcal{O}$-operators on  Malcev algebras and discuss the solutions of Malcev Yang-Baxter equation
by  $\mathcal{O}$-operators. Furthermore we introduce the notion of  weighted   $\mathcal{O}$-operators  on  Malcev algebras, which can be characterized
by graphs of the semi-direct product Malcev algebra. Then  we introduce a new algebraic structure called  post-Malcev algebras.
Therefore, post-Malcev algebras can be viewed as the underlying
algebraic structures of  weighted   $\mathcal{O}$-operators  on  Malcev algebras. A post-Malcev algebra also gives rise to a new Malcev algebra.
Post-Malcev algebras are analogues for Malcev algebras of post-Lie algebras and fit into a bigger framework with a close relationship with post-alternative algebras. \\[0,2cm]
\end{abstract}

\newpage
\tableofcontents
\section{Introduction}
~~

Associative dialgebras were introduced by Loday and Pirashvili \cite{LodayPirashvili93} as the universal
enveloping algebras of Leibniz algebras (a noncommutative analogue of Lie algebras).
Loday also defined dendriform dialgebras in his study of algebraic $K$-theory \cite{Loday01}.
Dendriform dialgebras have two operations whose sum is an associative operation.
Moreover, the operads associated to associative dialgebras and dendriform dialgebras are Koszul dual.
Loday and Ronco introduced trialgebras, a generalization of dialgebras, and dual to them, dendriform trialgebras \cite{LodayRonco04}.
This means one can decompose an associative product into sums of two or three operations satisfying some compatibility conditions.

Aguiar \cite{Aguiar00} first noticed a relation between Rota-Baxter algebras and dendriform dialgebras.
He proved that an associative algebra with a Rota-Baxter operator $\mathcal{R}$ of weight zero is a dendriform dialgebra
relative to operations $ x \prec y = x\mathcal{R}(y)$ and $x \succ y = \mathcal{R}(x)y$.
Ebrahimi-Fard \cite{EbrahimiFard02} generalized this fact to Rota-Baxter algebras of arbitrary weight
and obtained dendriform dialgebras and dendriform trialgebras.
Universal enveloping Rota-Baxter algebras of weight $\lambda$ for dendriform dialgebras and trialgebras were defined by
Ebrahimi-Fard and Guo in \cite{EbrahimiFardGuo08}.

This idea of splitting associative operations gave rise to different algebraic structures with multiple binary operations
such as quadri-algebras, ennea-algebras, NS-algebras, dendriform-Nijenhuis algebras or octo-algebras.
These constructions can be put into the framework of (black square) products of nonsymmetric operads \cite{EbrahimiFardGuo05}.

Analogues of dendriform dialgebras, quadri-algebras and octo-algebras for Lie algebras \cite{BaiLiuNi10,LiuNiBai11},
Jordan algebras \cite{BaiHou12,HouNiBai13}, alternative algebras \cite{BaiNi} or Poisson algebras \cite{Aguiar00} have been obtained.
They can be regarded as splitting the operations in these latter algebras.

$\textbf{Malcev algebras and Rota-Baxter operators.}$
Rota-Baxter operators were introduced by G. Baxter \cite{baxter} in 1960 in the study of fluctuation theory in Probability.Rota-Baxter operators
$\mathcal{R}: A\to A$ are defined on an algebra $A$ by the identity
\begin{align}
\mathcal{R}(x)\mathcal{R}(y) = \mathcal{R}(\mathcal{R}(x)y + x\mathcal{R}(y) + \lambda xy),
\label{rotabaxterop}
\end{align}
for all $x, y\in A$, where $\lambda$ is a scalar, called the weight of $\mathcal{R}$. These operators were then further
investigated, by G.-C. Rota \cite{Rota}, Atkinson \cite{Atkinson}, Cartier \cite{Cartier} and others. Since the late 1990s, the study of Rota-Baxter operators has made great
progress in many areas, both in theory and in applications \cite{Aguiar, EbrahimiFardGuo05,EbrahimiFardGuo08,EbrahimiFardGuo09, Guo0,Guo1,Guo2}.

The notion of  Malcev algebras was introduced by Malcev \cite{malcev}, who called these objects Moufang-Lie algebras.
\begin{df}
  A $\textbf{Malcev algebra}$ is a non-associative algebra $A$ with an anti-symmetric multiplication $[\c,\c]$ that satisfies the Malcev identity for $x,y,z \in A$,
\begin{equation}
\label{maltsev}
J(x,y,[x,z]) = [J(x,y,z),x],
\end{equation}
where $J(x,y,z) = [[x,y],z] + [[z,x],y] + [[y,z],x]$ is the Jacobian.

Expanding the Jacobian, the Malcev identity \eqref{maltsev} is equivalent to Sagle's identity
for $x,y,z,t\in A:$
\[[[x, z], [y, t]] = [[[x, y], z], t] + [[[y, z], t], x] + [[[z, t], x], y] + [[[t, x], y], z].\]
 \end{df}
 In particular, Lie algebras are examples of Malcev algebras.   Malcev algebras play an important role in the geometry of smooth loops.   Just as the tangent
 algebra of a Lie group is a Lie algebra, the tangent algebra of a locally analytic Moufang loop is a Malcev algebra
 \cite{Kerdman,kuzmin2,malcev,nagy,sabinin}.  The reader is referred to \cite{gt,myung,okubo} for discussions about the relationships between Malcev algebras, exceptional Lie algebras and physics.

On the other hand, Madariaga \cite{Madariaga} first established the following connection from Rota-Baxter operators on Malcev algebras
to pre-Malcev algebras.
\begin{df}
A $\textbf{pre-Malcev algebra}$ is a vector space $A$ endowed with a bilinear
product $\rhd$ satisfying the following identity
 for $x, y, z, t \in A$,
 \begin{equation}\label{PM}
 [y, z] \rhd (x \rhd t)+ [[x, y], z] \rhd t+ y \rhd ([x, z] \rhd t- x \rhd (y \rhd (z \rhd t)) + z \rhd (x \rhd (y \rhd t))=0,
\end{equation}
where $[x,y]=x\rhd y-y\rhd x.$
 \end{df}
\begin{thm}
Let $(A,[\c,\c])$ be a Malcev algebra and $\mathcal{R}$ a Rota-Baxter operator. Then the binary operation
$$x\vartriangleright y=[\mathcal{R}(x),y]$$
defines a pre-Malcev algebras.
\end{thm}

Closely related to Malcev algebras are alternative algebras.  An alternative algebra is an algebra whose associator is an alternating function.  In particular, all associative algebras are alternative, but there are plenty of non-associative alternative algebras, such as the octonions.  Roughly speaking, alternative algebras are related to Malcev algebras as associative algebras are related to Lie algebras.  Indeed, as Malcev observed in \cite{malcev}, every alternative algebra $A$ is Malcev-admissible.  There are many Malcev-admissible algebras that are not alternative; see, e.g., \cite{myung}.  The reader is referred to \cite{tw} for applications of alternative algebras to projective geometry, buildings, and algebraic groups.

\textbf{Malcev Yang-Baxter equation and $\mathcal{O}$-operators.}
The notion of Malcev bialgebra was introduced by Vershinin in \cite{versh}  as an analogue of a Lie bialgebra
(also see \cite{goncharov}). A class of Malcev bialgebras (coboundary
cases) are obtained from the solutions of an algebraic equation in a Malcev
algebra, which is an analogue of the classical Yang-Baxter equation (CYBE)
in a Lie algebra (\cite{goncharov}). It is called Malcev Yang-Baxter equation (MYBE)
for convenience. The CYBE arose from the study of inverse scattering theory in 1980s. Later it was recognized as the "semi-classical
limit" of the quantum Yang-Baxter equation which was encountered by C. N. Yang in the
computation of the eigenfunctions of a one-dimensional fermion gas with delta function
interactions \cite{yang} and by R. J. Baxter in the solution of the eight vertex model in statistical
mechanics \cite{baxter1}. The study of the CYBE is also related to classical integrable systems and
quantum groups (see \cite{Chari} and the references therein).
An important approach in the study of the CYBE was through the interpretation of its
tensor form in various operator forms which proved to be effective in providing solutions of
the CYBE, in addition to the well-known work of Belavin and Drinfeld \cite{BelavinDrinfeld}.

In the case of Lie algebras, it was Semenov-Tian-Shansky who first introduced
certain operator form of the CYBE(\cite{Semenov}). Later Kupershmidt introduced
a notion of $\mathcal{O}$-operator of a Lie algebra as a generalization of (the operator
form of) the CYBE in a Lie algebra \cite{K2}. Nevertheless, the $\mathcal{O}$-operators of
Lie algebras play more interesting roles in the study of the CYBE. In fact, a
skew-symmetric solution of the CYBE is exactly a special $\mathcal{O}$-operator (associated
to the coadjoint representation) and more importantly, there are a kind
of algebraic structures behind the $\mathcal{O}$-operators of Lie algebras and the related
CYBE, namely, pre-Lie algebras, in the following sense: the $\mathcal{O}$-operators of
Lie algebras provide a direct relationship between Lie algebras and pre-Lie
algebras and in the invertible cases, they provide a necessary and sufficient
condition for the existence of a compatible pre-Lie algebra structure on a Lie
algebra, as an immediate consequence, there are some solutions of the CYBE
in certain Lie algebras obtained from pre-Lie algebras \cite{BaiNi}.



The notion of post-algebras goes back to Rosenbloom in \cite{Rosenbloom} (1942). An equivalent
formulation of the class of post-algebras was given by Rousseau in \cite{Rousseau} (1969, 1970)
which became a starting point for deep research.
Post-Lie algebras have been introduced by Vallette in 2007 \cite{Vallette}
in connection with the homology of partition posets and the study of Koszul
operads. However, J. L. Loday studied pre-Lie algebras and post-Lie algebras
within the context of algebraic operad triples, see for more details \cite%
{Loday01,LodayRonco04}. In the last decade, many works \cite{Burde16,Fard,sl2}
intrested in post-Lie algebra structures, motivated by the importance of
pre-Lie algebras in geometry and in connection with generalized Lie algebra
derivations.

Recently, Post-Lie algebras which are non-associative algebras played an
important role in different areas of pure and applied mathematics. They
consist of a vector space $A$ equipped with a Lie bracket $[\cdot ,\cdot ]$
and a binary operation $\triangleright $ satisfying the following axioms%
\begin{align}
x\triangleright \lbrack y,z]& =[x\triangleright y,z]+[y,x\triangleright z],
\\
\lbrack x,y]\triangleright z& =as_{\triangleright
}(x,y,z)-as_{\triangleright }(y,x,z).
\end{align}%
If the bracket $[\cdot ,\cdot ]$ is zero, we have exactly a pre-Lie
structure.
 It is worth to note that, in spite the post-Lie product does not yield a Lie bracket by antisymmetrization,
the bilinear product $\{\c,\c\} : A\otimes A\to A$, defined for all $x, y\in A$ by
\begin{align}
 \{x, y\} = x \vartriangleright y - y \vartriangleright x + [x, y].
\end{align}
defines on $A$ another Lie algebra structure.
The varieties of pre- and post-Lie algebras play a crucial role
in the definition of any pre and post-algebra through black Manin operads
product, see details in \cite{BBGN2012,GubKol2014}. Whereas pre-Lie algebras
are intimately associated with euclidean geometry, post-Lie algebras occur
naturally in the differential geometry of homogeneous spaces, and are also
closely related to Cartan's method of moving frames.
Ebrahimi-Fard, Lundervold and Munthe-Kaas \cite{Fard}
studied universal enveloping algebras of post-Lie algebras and the free
post-Lie algebra.

It has been found recently that in parallel to the close connection between Lie algebra and
associative algebra by taking the commutator bracket, there is a close connection between the
dendriform trialgebra and post-Lie algebra. Moreover, in parallel to regarding the dendriform trialgebra
as a splitting of the associative algebra, the post-Lie algebra is a splitting of the Lie algebra
in a precise sense that applies to all operads \cite{BaiNi}. Additionally, the splitting of operations
can be achieved by the application of the Rota-Baxter operator. So a Rota-Baxter operator on
an associative algebra (resp. Lie algebra) gives rise to a dendriform trialgebra \cite{LodayRonco04} (resp. post-Lie
algebra \cite{Burde16}).
Thus we have the following commutative diagram of categories (the arrows will go in the
opposite direction for the corresponding operads).
\begin{equation}\label{di:postlie}
\begin{array}{l}
\xymatrix{
\begin{array}{c}
\substack{\text{Dendriform}\\ \text{Trialg.}}
\end{array}
\ar[rr]^{x\prec y + x\succ y + x\cdot y}
\ar[dd]^{x\vartriangleright y =x\succ y - y\prec x }_{[x,y]=x\cdot y -y\cdot x} &&
\begin{array}{c}
\substack{\text{Associative}\\ \text{alg.}}
\end{array}
\ar[dd]^{x\star y - y\star x} \\
&&\\
\begin{array}{c}
\substack{\text{Post-Lie}\\ \text{alg.}}
\end{array}
\ar[rr]^{x\vartriangleright y-y\vartriangleright x+[x,y]} &&
\begin{array}{c}
\substack{\text{Lie} \\ \text{alg.}}
\end{array}}
\end{array}
\end{equation}
In this paper we use   weighted $\mathcal{O}$-operators to split operations, although a generalization exists in the alternative setting in terms of bimodules. Diagram \eqref{eq:diagb} summarizes the results of the present work.
\begin{align}
\label{eq:diagb}
\begin{array}{l}
\xymatrix@!0{
&&&
\begin{array}{c}
\substack{\text{Post-Malcev} \\ \text{alg.}}
\end{array}
\ar@{-}[rrrrrr]\ar@{-}'[d][dd]
&&& &&&
\begin{array}{c}
\substack{\text{Post-alternative} \\ \text{alg.}}
\end{array}
\ar@{-}[dd]
\\
\begin{array}{c}
\substack{\text{Post-Lie} \\ \text{alg.}}
\end{array}
\ar@{-}[urrr]\ar@{-}[rrrrrr]\ar@{-}[dd]
&&&&& &
\begin{array}{c}
\substack{\text{Dendriform} \\ \text{trialg.}}
\end{array}
\ar@{-}[urrr]\ar@{-}[dd]&&&
\\
&&&
\begin{array}{c}
\substack{\text{Malcev} \\ \text{alg.}}
\end{array}
\ar@{-}'[rrr][rrrrrr]
&&& &&&
\begin{array}{c}
\substack{\text{Alternative}\\ \text{alg.}}
\end{array}
\\
\begin{array}{c}
\substack{\text{Lie} \\ \text{alg.}}
\end{array}
\ar@{-}[rrrrrr]\ar@{-}[urrr]
&&&&&&
\begin{array}{c}
\substack{\text{Associative} \\ \text{alg.}}
\end{array}
\ar@{-}[urrr] &&&&
}
\end{array}
\end{align}



In Section 2, we study the relationship between  $\mathcal{O}$-operators and Malcev Yang-Baxter equation.
We construct in section 3 alternative algebras structure associated to any post-alternative algebra. The multiplication is
given by
$$x\star y = x \prec y + y \succ x + x\c y .$$
In addition, we investigate the notion of a weighted $\mathcal{O}$-operator   to construct a post-alternative algebra structure on the $A$-bimodule $\mathbb{K}$-algebra of an alternative algebra $(A, \c)$.
The Last section is devoted to introduce the notion of post-Malcev algebra and we show that  weighted   $\mathcal{O}$-operators can be used to construct post-Malcev algebras.
We also reveal
a relation between post-Malcev algebras and post-alternative algebras by the
commutative diagram \eqref{di:postmalcev}.

$\textbf{Convention and notation.}$
Throughout this paper, all algebras are finite-dimensional and over a field  $\mathbb{K}$ of
characteristic zero.

Let $(A, \c)$ be an algebra. We use the following notations.
\begin{description}[labelindent=5pt,style=multiline,leftmargin=1cm]
  \item[(a)] Denote the left and right multiplication operator by $L_{\c}(x)$ and $R_{\c}(x)$  respectively, that is,
$L_{\c}(x)y = R_{\c}(y)x = x\c y$ for any $x, y\in A$. We also simply denote them by $L(x)$ and $R(x)$
respectively without confusion. In particular, if $(A, [\c, \c])$ is a Malcev algebra, we let $ad_{[\c,\c ]}(x) = ad(x)$
denote the adjoint operator, that is, $ad_{[\c,\c]}(x)y = ad(x)y = [x, y]$ for any $x, y\in A$.
\item[(b)] Let $r=\sum\limits_i x_i\otimes y_i\in A^{\otimes 2}$, and  denote
\begin{equation}
r_{12}=\sum_ix_i\otimes y_i\otimes 1,\quad r_{13}=\sum_{i}x_i\otimes
1\otimes y_i,\quad r_{23}=\sum_i1\otimes x_i\otimes y_i,
\label{eq:r12}
\end{equation}
where 1 is a unit element if $(A,\c)$ is unital or a symbol playing a similar role of
the unit for the nonunital cases. 
\item[(c)] Let $V$ be a vector space. The {\bf twisting operator}
$\sigma: V^{\otimes 2} \rightarrow V^{\otimes 2}$ is defined for all $a, b\in V$ by
$$
\sigma (a \otimes b) = b\otimes a.
$$
We call $r=\sum\limits_i a_i\otimes b_i\in
V^{\otimes 2}$ {\bf skew-symmetric} (resp. {\bf symmetric}) if
$r=-\sigma(r)$ (resp. $r=\sigma(r)$). On the other hand,  $r$ can be regarded as a linear map from  the dual space $V^*$ to $V$ in the following way, for all $a^*,b^*\in V^*$,
\begin{equation}
\langle a^*, r(b^*)\rangle =\langle a^*\otimes b^*,r\rangle, \label{eq:idenrmap}
\end{equation}
where $\langle\cdot , \cdot\rangle$ is the canonical pairing between $V^{*}$ and $V$.
\item[(d)] Let $V_1, V_2$ be two vector spaces and $T : V_1\to V_2$ be a linear map. Denote the dual (linear) map by
$T^{*} : V^{*}_{2} \to V^{*}_{1}$
 defined for all $a_1\in V_1, a^{*}_{2}\in V^{*}_{2}$ by
 \begin{equation}
\langle T^{*}(a^{*}_{2}),a_1\rangle	 = \langle a^{*}_{2},T (a_1)\rangle.
\label{eq:dualTmap}
\end{equation}
On the other hand, $T$ can be identified as an element in $V_{2} \otimes V^{*}_{1}$ by
\begin{equation}
\langle a_{2}^*\otimes a_{1},T\rangle=\langle a_{2}^{*}, T(a_{1})\rangle,\quad
\forall  a_1\in V_1, a^{*}_{2}\in V^{*}_{2}.\label{eq:idenTmap}
\end{equation}
Note that \eqref{eq:idenrmap} is exactly the case $V_1 = V^{*}_{2}$ of \eqref{eq:idenTmap}. Moreover, in
the above sense, any linear map $T : V_1 \to V_2$ is obviously an element in
$(V_{2} \oplus V^{*}_{1})\otimes(V_{2} \oplus V^{*}_{1})$.
 \item[(e)] Let $A$ be an algebra and $V$ be a vector space. For any linear map $\rho : A\rightarrow End(V)$, define a linear map $\rho^{*} : A\to End(V^{*})$, for all
$x\in A, a^*\in V^*, b\in V$, by
\begin{equation}
\langle \r^*(x)a^*,b\rangle=-\langle a^*,\r(x)b\rangle. \label{eq:dual}
\end{equation}
Note that in this case,  $\rho^{*}$
is different from the one given by   \eqref{eq:dualTmap} which
regards $End(V )$ as a vector space, too.
\end{description}
\section{ \texorpdfstring{$\mathcal{O}$}{}-operators and Malcev Yang-Baxter equation}
In this section,
we  recall the classical result that a skew-symmetric solution
of CYBE in a Malcev algebra gives an $\mathcal{O}$-operator through a duality between tensor product and
linear maps. Not every $\mathcal{O}$-operator on a Lie algebra  comes from a solution of CYBE in this way. However, any
$\mathcal{O}$-operator can be recovered from a solution of CYBE in a larger Malcev algebra.

We first
recall the concept of a representation $($see \cite{Kuzmin}$)$ and construct the dual representation.
\begin{df}[\cite{Kuzmin}] \label{df:repmalcevalgebra}
A $\textbf{representation}$ (or $\textbf{a module}$) of a Malcev algebra $(A, [\cdot, \cdot])$ on
a vector space $V$ is a linear map $\rho: A\longrightarrow End(V)$ such that, for all
  $x,y\in A$,
\begin{equation}\label{representation}
\rho([[x, y], z]) = \rho(x)\rho(y)\rho(z) - \rho(z)\rho(x)\rho(y) + \rho(y)\rho([z, x]) - \rho([y, z])\rho(x).
\end{equation}
We denote this  representation  by $(V, \rho)$.
\end{df}
 For all $x,y\in A$, define  the map $ad:A \longrightarrow End(A)$ by
$
ad_{x}(y)=[x,y].
$
Then $ad$ is a representation of the Malcev algebra $(A,[\cdot,\cdot])$ on $A$,  which is called the adjoint representation.

\begin{prop}
\label{pp:dual}
Let $(V,\r)$ be a representation  of a Malcev algebra  $(A,[\c,\c])$.
 Then $(V^*,\r^*)$ defined by    \eqref{eq:dual} is a representation of $A$, which is called the dual representation of $(V,\r)$.
\end{prop}
\begin{proof}
By   \eqref{representation}, we have, for $x,y,z\in A$,
\begin{equation*}
\rho([[y, x], z])=-\rho([[x, y], z]) = \rho(z)\rho(y)\rho(x) - \rho(y)\rho(x)\rho(z) - \rho(x)\rho([z, y]) + \rho([x, z])\rho(y).
\end{equation*}
So, for any $x,y,z\in A$, $a^{*}\in V^{*}$, $ b\in V$, we have
\begin{align*}
&\Big\langle \r^*([[x,y],z])a^*,b\Big\rangle =-\Big\langle a^*,\r([[x,y],z])b\Big\rangle=-\Big\langle a^*,-\r([[y,x],z])b\Big\rangle\\
&=-\Big\langle a^*,\big(\rho(y)\rho(x)\rho(z) - \rho(z)\rho(y)\rho(x) + \rho(x)\rho([z, y]) - \rho([x, z])\rho(y)\big)b\Big\rangle\\
&=-\Big\langle \big(-\rho^{*}(z)\rho^{*}(x)\rho^{*}(y) + \rho^{*}(x)\rho^{*}(y)\rho^{*}(z) + \rho^{*}([z, y])\rho^{*}(x) - \rho^{*}(y)\rho^{*}([x, z])\big)a^*,b\Big\rangle.
\end{align*}
Hence, since $\langle\cdot,\cdot\rangle$ is nondegenerate, we obtain
\begin{equation*}
\r^*([[x,y],z])=\rho^{*}(z)\rho^{*}(x)\rho^{*}(y) - \rho^{*}(x)\rho^{*}(y)\rho^{*}(z) + \rho^{*}([y, z])\rho^{*}(x) - \rho^{*}(y)\rho^{*}([z, x]).
\qedhere
\end{equation*}
\end{proof}


\begin{df}
Let $(A, [\c,\c])$ be a Malcev algebra and $r\in A \otimes A$. $r$ is called a solution of
Malcev Yang-Baxter equation in $A$ if $r$ satisfies
\begin{equation}
[r_{12},r_{13}]+[r_{12},r_{23}]+[r_{13},r_{23}]= 0,
\label{rslt}
\end{equation}
where
\begin{align*}
[r_{12},r_{13}]&=\sum_{i,j}[x_i,x_j]\otimes y_i\otimes
y_j,\\ [r_{13}, r_{23}]&=\sum_{i,j}x_i\otimes x_j\otimes
[y_i, y_j],\\ [r_{12}, r_{23}]&=\sum_{i,j}x_i\otimes
[y_i,  x_j]\otimes y_j.
\end{align*}
\end{df}
Equation \eqref{rslt} gives the tensor form of  Malcev Yang-Baxter equation. What we will do next is
to replace the tensor form by a linear operator satisfying some conditions.
\begin{thm}\label{cybe==dualoop}
Let $(A, [\c,\c])$ be a Malcev algebra and $r\in A \otimes A$. Then $r$ is
a skew-symmetric solution of Malcev Yang-Baxter equation in $A$ if and only if  $r$ satisfies for all $x^{*}, y^{*}\in A^{*}$,
\begin{equation}
[r(x^{*}), r(y^{*})] = r(ad^{*}r(x^{*})(y^{*}) - ad^{*}r(y^{*})(x^{*})).
\label{rOp}
\end{equation}
\end{thm}

\begin{proof}
 Let $\{e_i,...,e_n\}$ be a basis of $A$ and
$\{e_i^*,...,e_n^*\}$ be its dual basis. Suppose that $[e_i,
e_j]=\sum\limits_p c_{ij}^p e_p$ and $r=\sum_{i,j}{a_{ij}e_i\otimes
e_j}$. Hence $a_{ij}=-a_{ji}$.
Now, we have
\begin{align*}
\left[r_{12},r_{13}\right]&=\Big[\sum_{i,j}a_{ij}e_i\otimes
e_j\otimes 1,\sum_{k,l}a_{kl}e_k\otimes 1\otimes
e_l\Big]
=\sum_{i,j,k,l,p}a_{ij}a_{kl}c_{ik}^{p} e_p\otimes e_j\otimes e_l,
\\
\left[r_{13},r_{23}\right] &=\Big[\sum_{i,j}a_{ij}e_i\otimes 1\otimes
e_j,\sum_{k,l}a_{kl}1\otimes e_k\otimes
e_l\Big]=\sum_{i,j,k,l,p}a_{ij}a_{kl}c_{jl}^p e_i \otimes e_k\otimes e_p,
\\
\left[r_{12},r_{23}\right]&=\Big[\sum_{i,j}a_{ij}e_i\otimes e_j\otimes 1,\sum_{k,l}a_{kl}1\otimes e_k\otimes
e_l\Big] =\sum_{i,j,k,l,p}a_{ij}a_{kl}c_{jk}^p e_i \otimes e_p\otimes e_l.
\end{align*}
Then $r$ is a solution of the
Malcev Yang-Baxter equation in $(A,[\c,\c])$ if and only if (for
any $j,p,l$)
$$\sum_{i,k}\Big({a_{ij}a_{kl}c_{ik}^p+ a_{kp}a_{ij}c_{ki}^l+ a_{pi}a_{kl}c_{ik}^j}\Big)e_p\otimes e_j\otimes e_l=0.$$
On the other hand, by \eqref{eq:idenrmap}, we get
$ r(e_j^*)=\sum\limits_i a_{ij}e_i.$
Then, if we set $x^{*}=e_j^{*}$ and $y^{*}=e_l^{*}$, by   \eqref{rOp},
$$\sum_{i,k}\Big({a_{ij}a_{kl}c_{ik}^p+ a_{kp}a_{ij}c_{ki}^l+ a_{pi}a_{kl}c_{ik}^j}\Big)e_p=0.$$
Therefore, it is easy to see that $r$ is a solution of  Malcev Yang-Baxter equation in $A$ if and only
if $r$ satisfies \eqref{rOp}.
\end{proof}
\begin{df}
 Let $(A, [\c,\c])$ be a Malcev algebra. A symmetric bilinear form
$B$ on $A$ is called $\emph{invariant}$ if,
for all $x, y, z \in A$,
\begin{equation}
 B([x, y], z) = B(x, [y, z]).
 \end{equation}
\end{df}
\begin{df}
Let $(A, [\c,\c])$ be a Malcev algebra.
 A Rota-Baxter operator of  weight $0$ on  $A$ is
a linear map $\mathcal{R}: A\to A$ satisfying for all $x,y\in A,$
$$[\mathcal{R}(x),\mathcal{R}(y)] = \mathcal{R}([\mathcal{R}(x), y] + [x, \mathcal{R}(y)]).$$
\end{df}
\begin{cor}
Let $(A, [\c,\c])$ be a Malcev algebra and $r\in A \otimes A$. Assume $r$ is skew-symmetric
and there exists a nondegenerate symmetric invariant bilinear form $B$ on $A$. Define a linear
map $\varphi : A\to A^{\ast}$ by $\langle \varphi(x),y\rangle = B(x, y)$ for any $x, y\in A$. Then $r$ is a solution of the Malcev
Yang-Baxter equation in $A$ if and only if  $\mathcal{R} = r\varphi : A\to A$ is a Rota-Baxter operator.
\end{cor}
\begin{proof}
For any $x, y\in A$, we have
$$
\langle \varphi(ad(x)y), z\rangle=B([x,y],z)=B(z,[x,y])=-B(y,[x,z])=\langle ad^{*}(x)\varphi(y),z\rangle.$$
Hence $\varphi(ad(x)y)=ad^{*}(x)\varphi(y)$ for any $x, y\in A$.
Let $x^{*}=\varphi(x)$, $y^{*}=\varphi(y)$, then by Theorem \ref{cybe==dualoop}, $r$ is a solution of the Malcev Yang-Baxter
equation in $A$ if and only if
$$
[r\varphi(x), r\varphi(y)] = [r(x^*),r(y^*)] = r(ad^{*}r(x^{*})(y^{*}) - ad^{*}r(y^{*})(x^{*})) = r\varphi\big([r\varphi(x),y]+[x,r\varphi(y)]\big).$$
Therefore the conclusion holds.
\end{proof}
Now, we introduce the notion of  $\mathcal{O}$-operator of a Malcev algebra.
\begin{df}
Let $(A, [\c,\c])$ be a Malcev algebra and let $(V,\rho)$ be a representation of $A$. A linear map $T : V \to  A$ is called an $\mathcal{O}$-operator  associated to  $\rho$ if
for all $a, b \in V,$
\begin{align}
[T (a), T (b)] = T \big(\rho(T (a))b - \rho(T (b))a\big).
 \end{align}
\end{df}
\begin{ex}
 Let $(A, [\c,\c])$ be a Malcev algebra. Then a Rota-Baxter operator
(of weight zero) is an $\mathcal{O}$-operator of $A$ associated to the adjoint representation
$(A, ad)$ and a skew-symmetric solution of Malcev Yang-Baxtezr equation in $A$ is an $\mathcal{O}$-operator of $A$
associated to the representation $(A^{\ast}, ad^{\ast})$.
\end{ex}
Let $(A, [\c,\c])$ be a Malcev algebra.  Let $\rho^*:A\rightarrow gl(V^*)$ be the dual representation
of the representation $\rho: A\rightarrow gl(V)$ of the Malcev
algebra $A$.
 A linear map $T : V\to A$ can be identified as an
element in $ A\otimes V^*\subset (A\ltimes_{\rho^*}V^*)\otimes (A\ltimes_{\rho^*}V)
$
as follows.
Let $\{e_1,\cdots,e_n\}$ be a basis of $A$. Let
$\{v_1,\cdots, v_m\}$ be a basis of $V$ and $\{ v_1^*,\cdots,
v_m^*\}$ be its dual basis, that is $v_i^*(v_j)=\delta_{ij}$. Set
$T(v_i)=\sum\limits_{j=1}^na_{ij}e_j, i=1,\cdots, m$. Since as
vector spaces,
${\rm Hom}(V,A)\cong A\otimes V^*,$ we
have
\begin{align}\label{Tidentification}
T&=\sum_{i=1}^m T(v_i)\otimes v_i^*
=\sum_{i=1}^m\sum_{j=1}^n a_{ij}e_j\otimes v_i^*\nonumber\\
&\in A\otimes V^*\subset (A\ltimes_{\rho^*}V^*)\otimes (A\ltimes_{\rho^*}V^*).
\end{align}
\begin{thm}\label{cybe==oop}
Let $(A, [\c,\c])$ be a Malcev algebra. Then $T$ is an $\mathcal{O}$-operator of $A$ associated to $(V,\rho)$
if and only if $r = T -\sigma(T)$ is a skew-symmetric solution of the  Malcev Yang-Baxter equation
in $A\ltimes_{\rho^*}V^{\ast}.$
\end{thm}

\begin{proof}
From \eqref{Tidentification}, we have
$ r = T -\sigma(T) =\sum _{i} T(v^{i})\otimes v^{\ast}_{i}-v^{\ast}_{i}\otimes T(v^{i}). $
Thus,
\begin{eqnarray*}
 &&[r_{12},r_{13}] =\sum_{i,k=1}^m
\{[T(v_i), T(v_k)]\otimes v_i^*\otimes v_k^* -\rho^*
(T(v_i))v_k^*\otimes v_i^*\otimes T(v_k)+ \rho^*
(T(v_k))v_i^*\otimes
T(v_i)\otimes v_k^* \},\\
&&[ r_{12},r_{23}] =\sum_{i,k=1}^m \{- v_i^*\otimes [T(v_i), T(v_k)]\otimes
v_k^*
- T(v_i)\otimes \rho^* (T(v_k))v_i^*\otimes v_k^*+ v_i^*\otimes \rho^*
(T(v_i))v_k^*\otimes
T(v_k)\},\\
&&[r_{13},r_{23}] =\sum_{i,k=1}^m \{v_i^*\otimes v_k^*\otimes [T(v_i),
T(v_k)]
+ T(v_i)\otimes v_k^*\otimes \rho^*(T(v_k))v_i^*- v_i^*\otimes T(v_k)\otimes
\rho^* (T(v_i))v_k^*\}.
\end{eqnarray*}
By the definition of dual representation, we know
$\rho^*(T(v_k))v_i^*=-\sum_{j=1}^m v_i^*(\rho(T(v_k))v_j) v_j^*.$
Thus,
\begin{eqnarray*}
&&-\sum_{i,k=1}^m T(v_i)\otimes \rho^*(T(v_k))v_i^*\otimes
v_k^*=-\sum_{i,k=1}^mT(v_i)\otimes
[\sum_{j=1}^m -v_i^*(\rho(T(v_k))v_j) v_j^*]\otimes v_k^*\\
&&=\sum_{i,k=1}^m \sum_{j=1}^m v_j^*(\rho(T(v_k))v_i) T(v_j)\otimes
v_i^*\otimes v_k^*
=\sum_{i,k=1}^m T(\sum_{j=1}^m (v_j^*(\rho(T(v_k))v_i) v_j)\otimes
v_i^*\otimes v_k^* \\
&&=\sum_{i,k=1}^m T(\rho (T(v_k))v_i)\otimes v_i^*\otimes v_k^*.
\end{eqnarray*}
Therefore,
\begin{eqnarray*}
&&[r_{12},r_{13}]+[r_{12},r_{23}]+[r_{13},r_{23}]\\
&&=\sum_{i,k=1}^m\{
([T(v_i),T(v_k)]+T(\rho(T(v_k))v_i)-T(\rho(T(v_i))v_k))\otimes v_i^*\otimes
v_k^*\\
&& -v_i^*\otimes ([T(v_i),T(v_k)]+T(\rho(T(v_k))v_i)-T(\rho(T(v_i))v_k))\otimes
v_k^*\\
&& + v_i^*\otimes v_k^*\otimes
([T(v_i),T(v_k)]+T(\rho(T(v_k))v_i)-T(\rho(T(v_i))v_k))\}.
\end{eqnarray*}
So $r$ is a classical $r$-matrix in  $A\ltimes_{\rho^*}V^*$
if and only if $T$ is an $\mathcal{O}$-operator.
\end{proof}

In fact, Theorem \ref{cybe==oop} gives a relation between  $\mathcal{O}$-operator and Malcev Yang-Baxter equation. Then, we get a direct conclusion from Theorems \ref{cybe==dualoop} and \ref{cybe==oop}.
\begin{cor}
Let $(A, [\c,\c])$ be a Malcev algebra.  Let $\rho:A\rightarrow gl(V)$ be a representation
 of $A$.
 Set
$\widehat{A} = A\ltimes_{\rho^*}V^*.$
  Let $T : V\to A$  be a linear map. Then the following three conditions
are equivalent:
\begin{enumerate}[label*={\rm (\roman*)}]
\item $T$ is a $\mathcal{O}$-operator of $A$ associated to $\rho$;
\item $T -\sigma(T)$ is a skew-symmetric solution of the Malcev Yang-Baxter equation in $\widehat{A}$;
\item $T -\sigma(T)$ is a $\mathcal{O}$-operator of the Malcev algebra $\widehat{A}$ associated to $ad^{*}.$
\end{enumerate}
\end{cor}
\section{Alternative and post-alternative algebras}
In this section, we recall some basic definitions about alternative and pre-alternative algebras studied  in  \cite{Kuzmin1,BaiNi}.
\subsection{Some basic results on alternative algebras}

\begin{df}
\label{def:altalg}
  An  $\textbf{alternative algebra}$ $(A,\c)$ is a vector space $A$ equipped with a
bilinear operation $(x,y)\rightarrow x\c y$ satisfying, for all $x,y,z\in A$,
\begin{equation}
as_{A}(x,x,y)=as_{A}(y,x,x)=0,
\end{equation}
where $as_{A}(x,y,z)=(x\c y)\c z-x\c(y\c z)$ is the
$\textbf{associator}$.
\end{df}

\begin{rem}
\label{remequivdef:altalg}
 If the characteristic of the field is not 2, then an alternative
algebra $(A,\c)$ also satisfies the stronger axioms, for all $x, y, z\in A$,
\begin{align}
\label{alt1} as_A (x, y, z) + as_A (y, x, z)& = 0,\\
\label{alt2} as_A (z, x, y) + as_A (z, y, x)& = 0.
\end{align}
\end{rem}
Now, recall that an algebra $(A, \c)$  is called admissible Malcev
algebra if $(A, [\c, \c])$  is a Malcev algebra, where $[x, y] = x \c y - y \c x$.
\begin{ex}
Any alternative algebra is Malcev admissible. That is if $(A,\c)$ be an
alternative algebra, then $(A, [\c,\c])$ is a Malcev algebra, where $[x, y] = x\c y - y\c x$, for
all $x, y\in A$.
\end{ex}
\begin{df}[\cite{Schafer}]
Let $(A, \c)$ be an alternative algebra and $V$ be a vector space.
Let $\mathfrak{l},\mathfrak{r}:A \to End(V)$ be two linear maps. Then, $(V,\mathfrak{l},\mathfrak{r})$ is called a representation or
a bimodule of $A$ if, for any $x, y \in A$,
\begin{align}&\mathfrak{r}(x)\mathfrak{r}(y)+\mathfrak{r}(y)\mathfrak{r}(x)-\mathfrak{r}(x\c y)-\mathfrak{r}(y\c x)=0,\\
 &\mathfrak{l}(x\c y)+\mathfrak{l}(y\c x)-\mathfrak{l}(x)\mathfrak{l}(y)-\mathfrak{l}(y)\mathfrak{l}(x)=0,\\
&\mathfrak{l}(x\c y)+ \mathfrak{r}(y)\mathfrak{l}(x)- \mathfrak{l}(x)\mathfrak{l}(y)-\mathfrak{l}(x)\mathfrak{r}(y)=0,\\ &\mathfrak{r}(y)\mathfrak{l}(x)+\mathfrak{r}(y)\mathfrak{r}(x)-\mathfrak{l}(x)\mathfrak{r}(z)-\mathfrak{r}(x\c y)=0.
\end{align}
\end{df}

\begin{df}
$\textbf{A pre-alternative algebra}$ is a triple $(A,\prec,\succ)$, where $A$ is a vector space,
$\prec,\succ: A \otimes A \rightarrow A$
are bilinear maps   satisfying
for all $x, y, z \in A$ and  $x \c y = x \prec y + x\succ y$,
\begin{eqnarray}
&&(x\succ y) \prec z - x\succ(y \prec z) + (y \prec x) \prec z - y \prec (x \c z) = 0,\\
&&(x\succ y)\prec z - x\succ(y \prec z) + (z \c x)\succ y - z\succ(x\succ y) = 0,\\
&&(x\c y) \succ z - x\succ(y \succ z) + (y \c x) \succ z - y \succ (x \succ z) = 0,\\
&&(z\prec x)\prec y - z\prec(x \star y) + (z \prec y)\prec x - z\prec(y\c x) = 0.
\end{eqnarray}
\end{df}
\begin{prop}
Let $(A,\prec,\succ)$ be a pre-alternative algebra. Then the product  $x \c y = x \prec y + x\succ y$ defines an alternative algebra $A$. Furthermore, $(A, L_\succ, R_\prec)$, where $L_\succ(x)y=x\succ y$ and $ R_\prec(x)y=y\prec x$, gives a representation  of
the associated alternative algebra $(A, \c)$ on $A$.
\end{prop}
\begin{prop}\label{prealt==>premal}
Let $(A, \prec , \succ)$ be a pre-alternative algebra.
Then the product
$x\rhd y=x\succ y-y\prec x$
defines a pre-Malcev structure in $A$.
\end{prop}
\subsection{\texorpdfstring{$A$}{}-bimodule alternative algebras, weighted \texorpdfstring{$\mathcal{O}$}{}-operators and post-alternative algebras}
\begin{df}
Let $(A,\c)$ be an alternative  algebra. Let $(V,\c_V)$ be an alternative algebra and
 $\mathfrak{l},\mathfrak{r}:A\to End(V)$ be two linear maps. We say that $(V,\c_V,\mathfrak{l},\mathfrak{r})$ is an $\textbf{A-bimodule alternative algebra}$ if $(V,\mathfrak{l},\mathfrak{r})$ is a representation of  $(A,\c)$  such that the following compatibility conditions hold $($for all $x \in A,\ a,b \in V)$
\begin{eqnarray}
\label{repKaltalgebras1}
 &&\mathfrak{r}(x)(a\c_V b)- a\c_V (\mathfrak{r}(x)b) +\mathfrak{r}(x)(b\c_V a) - b\c_V (\mathfrak{r}(x)a)=0,\\
\label{repKaltalgebras2} &&(\mathfrak{l}(x)a)\c_V b-\mathfrak{l}(x)(a\c_V b) +(\mathfrak{l}(x)b)\c_V a - \mathfrak{l}(x)(b\c_V a)=0,\\
\label{repKaltalgebras3} &&(\mathfrak{l}(x)a)\c_V b- a\c_V (\mathfrak{l}(x)b) +(\mathfrak{r}(x)a)\c_V b -\mathfrak{l}(x)(a\c_V b)=0,\\
\label{repKaltalgebras4} &&(\mathfrak{r}(x)a)\c_V b-a\c_V (\mathfrak{l}(x)b) +\mathfrak{r}(x)(a\c_V b) - a\c_V (\mathfrak{r}(x)b)=0.
 \end{eqnarray}
 \end{df}

\begin{prop}\label{semidirectproductalt}
Let $(A,\c)$ and $(V,\c_V)$ be two alternative algebras and
  $\mathfrak{l},\mathfrak{r}:A\to End(V)$ be two linear maps. Then $(V,\c_V,\mathfrak{l},\mathfrak{r})$ is an $A$-bimodule alternative algebra if and only if the direct sum $A\oplus V$ of vector spaces
 is an alternative algebra (the semi-direct sum) with the product
on $A\oplus V$ defined for all $x,y\in A, \ a,b\in V$ by
\begin{equation}
(x+a)\ast (y+b)=x\c y+\mathfrak{l}(x)b+\mathfrak{r}(y)a+a\c_V b.
\end{equation}
We denote this algebra by $A\ltimes_{\mathfrak{l},\mathfrak{r}} V$ or simply $A\ltimes V$. Further, if $(A,\c)$ is an alternative algebra, then it is easy to see that
$(A, \c, L_\c,R_\c)$ is an $A$-bimodule alternative algebra, where $L_\c$ and $R_\c$ are the
left and right multiplication operators corresponding to the multiplication $\c$.
\end{prop}
\begin{proof}
For any $x,y,z\in A, a,b,c\in V$
\begin{align*}
&as_{A\oplus V} (x+a, y+b, z+c) + as_{A\oplus V} (y+b, x+a, z+c)=\\
&((x+a) \ast (y+b)) \ast (z+c)-(x+ a) \ast ((y+b) \ast (z+c)) + ((y+b) \ast (x+a)) \ast (z+c) \\
&\quad - (y+b) \ast ((x+a) \ast (z+c))\\
&=(x\c y+\mathfrak{l}(x)b+\mathfrak{r}(y)a+a\c_V b)\ast (z+c)-(x+ a)\ast (y\c z +\mathfrak{l}(y)c+\mathfrak{r}(z)b+b\c_V c)\\
&\quad+(y\c x+\mathfrak{l}(y)a+\mathfrak{r}(x)b+b\c_V a)\ast (z+c)-(y+b) \ast (x\c z +\mathfrak{l}(x)c+\mathfrak{r}(z)a+a\c_V c)\\
&\quad-x\c(y\c z)-\mathfrak{l}(x)\big(\mathfrak{l}(y)c+\mathfrak{r}(z)b+b\c_V c\big) -\mathfrak{r}(y\c z)a-a\c_V \big(\mathfrak{l}(y)c+\mathfrak{r}(z)b+b\c_V c\big)\\
&\quad+(y\c x)\c z+\mathfrak{l}(y\c x)c+\mathfrak{r}(z)\big(\mathfrak{l}(y)a+\mathfrak{r}(x)b+b\c_V a\big) +\big(\mathfrak{l}(y)a+\mathfrak{r}(x)b+b\c_V a\big)\c_V c\\
&\quad-y\c(x\c z)-\mathfrak{l}(y)\big(\mathfrak{l}(x)c+\mathfrak{r}(z)a+a\c_V c\big)-\mathfrak{r}(x\c z)b- b\c_V \big(\mathfrak{l}(x)c+\mathfrak{r}(z)a+a\c_V c\big).
\end{align*}
Hence $as_{A\oplus V} (x+a, y+b, z+c) + as_{A\oplus V} (y+b, x+a, z+c)=0$ if and only if  \eqref{alt1}, \eqref{repKaltalgebras1} and \eqref{repKaltalgebras3} hold.\\
 Analogously,
$as_{A\oplus V} (z+c,x+a, y+b) + as_{A\oplus V} (z+c,y+b,x+a)= 0$
if and only if \eqref{alt2}, \eqref{repKaltalgebras2} and \eqref{repKaltalgebras4} hold.
\end{proof}

\begin{df}[\cite{BBGN2012}]\label{postaltr}
  A $\textbf{post-alternative algebra}$ $(A, \prec , \succ , \cdot  )$ is a
vector space $A$ equipped with bilinear operations $\prec , \succ , \cdot :A\otimes A\rightarrow A$
obeying the following equations for $\star= \prec+\succ+\c$  and all $x, y, z\in A$,
\begin{eqnarray}
&&(x \c y) \c z- x \c (y \c z) + (y \c x) \c z  - y \c(x \c z)=0, \label{postaltr1}\\
&&(z \c x)\c y- z \c (x\c y) + (z\c y) \c x -z\c (y \c x)=0, \label {postaltr2}\\
&&(x \c y) \prec z -x \c (y \prec z) + (y \c x) \prec z  - y \c (x \prec z)=0, \label{postaltr3}\\
&&(x \succ y) \c z- x \succ (y \c z) + (x \succ z) \c y  - x \succ (z \c y)=0, \label{postaltr4}\\
&&(y \succ x) \c z- x \c (y \succ z)+ (x \prec y) \c z  - y \succ (x \c z)=0, \label{postaltr5}\\
&&(z \prec x)\c y- z \c (x \succ y)+(z\c y) \prec x - z \c (y \prec x)  =0, \label{postaltr6}\\
&&(x\succ y) \prec z - x\succ(y \prec z) + (y \prec x) \prec z - y \prec (x \star z) = 0,\label{postaltr7}\\
&&(x\succ y)\prec z - x\succ(y \prec z) + (z \star x)\succ y - z\succ(x\succ y) = 0,\label{postaltr8}\\
&&(x\star y) \succ z - x\succ(y \succ z) + (y \star x) \succ z - y \succ (x \succ z) = 0,\label{postaltr9}\\
&&(z\prec x)\prec y - z\prec(x \star y) + (z \prec y)\prec x - z\prec(y\star x) = 0.\label{postaltr10}
\end{eqnarray}
\end{df}
\begin{rem}
 Let $(A, \prec , \succ , \cdot  )$ be a post-alternative algebra. If the operation $\c$ is trivial, then it
is a pre-alternative algebra.
\end{rem}

Let $(A, \prec , \succ , \cdot  )$ be a post-alternative algebra, it is obvious that $(A,\c )$ is an alternative algebra.
On the other hand, it is straightforward to get the following conclusion:
\begin{thm}
If $(A, \prec , \succ , \cdot  )$ is a post-alternative algebra, then with a new bilinear operation
$\star: A\times A\to A$ on $A$ defined for all $x,y\in A$ by
\begin{equation}
x\star y= x\prec y+ x\succ y+ x\c y,
\end{equation}
$(A, \star)$ becomes an alternative algebra. It is called the associated alternative algebra of $(A, \prec , \succ , \cdot  )$.
\end{thm}
\begin{proof}
In fact, for any $x,y,z\in A$, we have
\begin{align*}
& as_A (x,y,z)+ as_A (y,x,z)=
(x\star y)\star z-x\star(y\star z)+ (y\star x)\star z-y\star(x\star z)\\
 &=(x\star y)\prec z+(x\star y)\succ z+(x\star y)\c z -x\prec(y\star z)-x\succ(y\star z)-x\c(y\star z)\\
 &\quad+(y\star x)\prec z+(y\star x)\succ z+(y\star x)\c z-y\prec(x\star z)-y\succ(x\star z)-y\c(x\star z)\\
 &=(x\prec y)\prec z+(x\succ y)\prec z+(x\c y)\prec z+(x\star y)\succ z+(x\prec y)\c z+(x\succ y)\c z\\
 &\quad+(x\c y)\c z-x\prec(y\star z)-x\succ(y\prec z)- x\succ(y\succ z)-x\succ(y\c z) -x\c(y\prec z)\\
&\quad- x\c(y\succ z)-x\c(y\c z)+(y\prec x)\prec z+(y\succ x)\prec z+(y\c x)\prec z+(y\star x)\succ z\\
&\quad+(y\prec x)\c z+(y\succ x)\c z +(y\c x)\c z-y\prec(x\star z)-y\succ(x\prec z)- y\succ(x\succ z)\\
&\quad-y\succ(x\c z) -y\c(x\prec z)- y\c(x\succ z)-y\c(x\c z) =0,
\end{align*}
and then replacing $(x,y,z)$ in this computation by $(z,x,y)$ yields
$as_A (z,x,y)+ as_A (z,y,x)=0$, which completes the proof according to Definition \ref{def:altalg} and Remark \ref{remequivdef:altalg}.
\end{proof}
The following terminology is motivated by the notion of $\lambda$-weighted $\mathcal{O}$-operator as a generalization of
(the operator form of) the classical Yang-Baxter equation in \cite{Bai2,K2}.
\begin{df}\label{o-opALT}
Let $(A, \cdot)$ be an alternative algebra and $(V,\c_V,\mathfrak{l},\mathfrak{r})$ be an $A$-bimodule alternative algebra. A linear map $T : V \to  A $ is called a $\lambda$-weighted $\mathcal{O}$-operator associated to $(V,\c_V,\mathfrak{l},\mathfrak{r})$
 if $T$ satisfies, for all $a, b \in V$,
\begin{equation}\label{oopaltalg}
 T (a) \cdot T (b)= T \big(\mathfrak{l}(T(a))b + \mathfrak{r}(T (b))a+ \lambda a\c_V b \big).
\end{equation}
\end{df}
When $(V,\c_V,\mathfrak{l},\mathfrak{r}) = (A,\c ,L_\c,R_\c)$, the condition \eqref{oopaltalg}  becomes
\begin{equation}\label{rotabaxteraltalg}
\mathcal{R}(x)\c \mathcal{R}(y)=\mathcal{R}\big(\mathcal{R}(x)\c y+ x\c \mathcal{R}(y) + \lambda x\c y\big).
\end{equation}
The property \eqref{rotabaxteraltalg} implies that $\mathcal{R}:A\to A$ is a $\lambda$-weighted Rota-Baxter operator  on the alternative algebra $(A, \c)$.

\begin{thm}
Let $(A, \cdot)$ be an alternative algebra and $(V,\c_V,\mathfrak{l},\mathfrak{r})$ be an $A$-bimodule alternative algebra. Let $T:V \to A$ be a $\lambda$-weighted $\mathcal{O}$-operator  associated to $(V,\c_V,\mathfrak{l},\mathfrak{r})$.
Define three new bilinear operations $\prec,\succ,\circ
:V\otimes V\to V$ on $V$ as follows:
\begin{equation}
  \label{alt==>postalt} a\succ b= \mathfrak{l}(T (a))b,\quad a\prec b= \mathfrak{r}(T (b))a, \quad a\circ b= \lambda a\c_V b.
\end{equation}
Then $(V,\prec,\succ, \circ)$  becomes a post-alternative algebra and $T$ is a homomorphism of alternative algebras.
\label{thm:depostalt}
\end{thm}
\begin{proof}
Since $A$ is an alternative algebra, \eqref{postaltr1} and \eqref{postaltr2} obviously hold. Furthermore, for any
$a,b,c\in V$, we have
\begin{align*}
&(a \circ b) \prec c -a \circ (b \prec c) + (b \circ a) \prec c  - b \circ (a \prec c)\\
&=(\lambda a\c_V b)\prec c-a \circ (\mathfrak{r}(T(c))b)+(\lambda b\c_V a)\prec c- b \circ (\mathfrak{r}(T(c))a)\\
&=\lambda\big((\mathfrak{r}(T(c))a)\c_V b-a\c_V (\mathfrak{r}(T(c))b)+(\mathfrak{r}(T(c))b)\c_V a-b\c_V (\mathfrak{r}(T(c))a) \big)=0.
\end{align*}
So, \eqref{postaltr3} holds. Moreover, \eqref{postaltr4} holds. Indeed,
\begin{align*}
    &(a \succ b) \circ c- a \succ (b \circ c) + (a \succ c) \circ b  - a \succ (c \circ b)\\
    &=(\mathfrak{l}(T (a))b) \circ c- a \succ (\lambda b\c_V c) + (\mathfrak{l}(T (a))c)\circ b
    - a \succ (\lambda c\c_V b)\\
    &= \lambda\big((\mathfrak{l}(T (a))b)\c_V c- \mathfrak{l}(T (a))(b\c_V c)  + (\mathfrak{l}(T (a))c)\c_V b- \mathfrak{l}(T (a))(c\c_V b)\big)=0.
\end{align*}
To prove identity \eqref{postaltr5}, we compute as follows
\begin{align*}
&(b \succ a) \circ c- a \circ (b \succ c)+ (a \prec b) \circ c  - b \succ (a \circ c)\\
&= (\mathfrak{l}(T (b))a)\circ c- a \circ (\mathfrak{l}(T (b))c)+(\mathfrak{r}(T (b))a)\circ c  - b \succ (\lambda a \c_V  c)\\
&=\lambda\big(\mathfrak{l}(T (b))a)\c_V c- a \c_V (\mathfrak{l}(T (b))c)+(\mathfrak{r}(T (b))a)\c_V c-\mathfrak{l}(T (b))(a \c_V  c)\big)=0.
\end{align*}
The other identities can be shown similarly.
\end{proof}
\begin{cor}
\label{cor:postaltRB}
Let $(A, \c)$ be an alternative
algebra and $\mathcal{R}: A\rightarrow A$ be a $\lambda$-weighted Rota-Baxter operator for $A$.
Then $(A,\prec,\succ,\circ)$ is a post-alternative algebra with the operations
    \begin{eqnarray*}
        x \prec y = x \c \mathcal{R}(y),\quad
        x \succ y = \mathcal{R}(x) \c y,
        \quad x \circ y = \lambda x \c y.
\end{eqnarray*}
\end{cor}
\section{Weighted \texorpdfstring{$\mathcal{O}$}{}-operators and post-Malcev algebras}
We start this section by introducing the notion of post-Malcev algebra  together with some of
its basic properties. We will also briefly discuss the post-Malcev
algebra structure underneath the $\lambda$-weighted $\mathcal{O}$-operators. We then show that there is a close relationship between post-Malcev algebras and post-alternative algebras in parallel to the relationship between pre-Malcev and pre-alternative algebras.
\subsection{\texorpdfstring{$A$}{}-module Malcev algebras and weighted \texorpdfstring{$\mathcal{O}$}{}-operators }
Now, we  extend the concept of a module to that of an
$A$-module algebra by replacing the $\mathbb{K}$-module $V$ by a
Malcev algebra. Next, we introduce $\lambda$-weighted $\mathcal{O}$-operators on Malcev algebras and study
some basic properties.

\begin{df}
Let $(A,[\c,\c])$ and $(V,[\c,\c]_V)$ be two Malcev algebras. Let
 $\rho:A  \to End(V)$ be a linear map such that  $(V,\rho)$ is a representation of  $(A,[\c,\c])$  and the following compatibility conditions hold for all $x, y, \in A,\ a,b, c \in V:$
\begin{align}\label{repKalgebras1}
&\r([x, y])[a,b]_V =\r(x)[\r(y)a, b]_V- [\r(y)\r(x)a, b]_V - [\r(x)\r(y)b, a]_V + \r(y)[\r(x)b, a]_V ,\\
\label{repKalgebras2} &[\r(x)a, \r(y)b]_V = [\r([x, y])a, b]_V - \r(x)[\r(y)a, b]_V + \r(y)\r(x)[a, b]_V + [\r(y)\r(x)b, a]_V,\\
\label{repKalgebras3} &[\r(x)a, [b, c ]_V ]_V= [[\r(x)b, a]_V, c]_V - \r(x)[[b, a]_V,c ]_V  - [\r(x)[a, c]_V , b]_V - [[\r(x)c, b]_V, a]_V.
\end{align}
Then $(V,[\c,\c]_V,\rho)$ is called an $\textbf{A-module Malcev algebra}$.
\end{df}
In the sequel, an $A$-module Malcev algebra is denoted by $(V; [\c,\c]_V,\r)$. It is straightforward to get the following:
\begin{prop}
 Let $(A,[\c,\c])$ and $(V,[\c,\c]_V)$ be two  Malcev algebras and $(V; [\c,\c]_V,\rho)$ be an $A$-module Malcev algebra. Then $(A\oplus V, [\c,\c]_\r) $ carries a new Malcev algebra structure with bracket
\begin{align} \label{semidirectproductMalcev}
[x+a,y+b]_{\rho}=&[x,y]+\rho(x)b-\rho(y)a+[a,b]_V, \quad \forall x,y\in A,\quad a,b\in V.
\end{align}
This is called the semi-direct product, often denoted by  $A\ltimes_{\r} V$ or simply $A\ltimes V$.
\end{prop}
\begin{proof}
For $x,y,z,t\in A$ and $a,b,c,d\in V$,
  \begin{align*}
  &
  [[x+ a, z + c]_\r, [y + b, t + d]_\r]_\r=
  [[x, z], [y, t]] + \r([x, z])\r(y)d -  \r([x, z])\r(t)b\\
  &\quad + \r([x, z])[b, d]_V - \r([y, t])\r(x)c +  \r([y, t])\r(z)a - \r([y, t])[a, c]_V + [\r(x)c, \r(y)d]_V\\
   & \quad-[\r(x)c, \r(t)b]_V  + [\r(x)c, [b, d]_V ]_V- [\r(z)a, \r(y)d]_V + [\r(z)a, \r(t)b]_V - [\r(z)a, [b, d]_V ]_V\\
   & \quad+ [[a, c]_V, \r(y)d]_V -[[a, c]_V, \r(t)b]_V + [[a, c]_V, [b, d]_V]_V ,
\\[0.1cm]
   &[[[x + a, y + b]_\r, z+ c]_\r, t+ d]_\r= [[[x, y], z], t] + \r([[x, y], z])d - \r(t)\r([x, y])c\\
   &\quad +  \r(t)\r(z)\r(x)b - \r(t)\r(z)\r(y)a + \r(t)\r(z)[a, b]_V - \r(t)[\r(x)b, c]_V + \r(t)[\r(y)a, c]_V \\
   &\quad- \r(t)[[a, b]_V,c ]_V + [\r([x, y])c, d]_V - [\r(z)\r(x)b, d]_V + [\r(z)\r(y)a, d]_V - [\r(z)[a, b]_V, d]_V\\
   &\quad + [[\r(x)b, c]_V, d]_V - [[\r(y)a, c]_V, d]_V + [[[a, b]_V, c]_V, d]_V,
\\[0.1cm]
    &[[[y + b, z + c]_\r, t+ d]_\r, x+ a]_\r=[[[y, z], t], x] + \r([[y, z], t])a - \r(x)\r([y, z])d\\
    &\quad +  \r(x)\r(t)\r(y)c -  \r(x)\r(t)\r(z)b + \r(x)\r(t)[b, c]_V - \r(x)[\r(y)c, d]_V + \r(x)[\r(z)b, d]_V\\
    &\quad - \r(x)[[b, c]_V, d]_V + [\r([y, z])d, a]_V - [\r(t)\r(y)c, a]_V + [\r(t)\r(z)b, a]_V - [\r(t)[b, c]_V, a]_V\\
    &\quad + [[\r(y)c, d]_V, a]_V- [[\r(z)b, d]_V, a]_V + [[[b, c]_V, d]_V, a]_V,
\\[0.1cm]
    &[[[z + c, t + d]_\r, x+ a]_\r, y+ b]_\r=[[[z, t], x], y] + \r([[z, t], x])b - \r(y)\r([z, t])a\\
    &\quad +  \r(y)\r(x)\r(z)d -  \r(y)\r(x)\r(t)c + \r(y)\r(x)[c, d]_V - \r(y)[\r(z)d, a]_V + \r(y)[\r(t)c, a]_V\\
    &\quad - \r(y)[[c, d]_V, a]_V + [\r([z, t])a, b]_V - [\r(x)\r(z)d, b]_V + [\r(x)\r(t)c, b]_V - [\r(x)[c, d]_V , b]_V\\
    &\quad + [[\r(z)d, a]_V, b]_V - [[\r(t)c, a]_V , b]_V + [[[c, d]_V, a]_V , b]_V,
\\[0.1cm]
    &[[[t + d, x + a]_\r, y+ b]_\r, z+ c]_\r=[[[t, x], y], z] + \r([[t, x], y])c - \r(z)\r([t, x])b\\
    &\quad +  \r(z)\r(y)\r(t)a -  \r(z)\r(y)\r(x)d + \r(z)\r(y)[d, a]_V - \r(z)[\r(t)a, b]_V + \r(z)[\r(x)d, b]_V\\
    &\quad - \r(z)[[d, a]_V, b]_V + [\r([t, x])b, c]_V - [\r(y)\r(t)a, c]_V + [\r(y)\r(x)d, c]_V - [\r(y)[d, a]_V, c]_V\\
    &\quad + [[\r(t)a, b]_V, c]_V - [[\r(x)d, b]_V , c]_V + [[[d, a]_V, b]_V , c]_V.
    \end{align*}
    Then $A\oplus V$ is  a  Malcev algebra if and only if $(V,\rho)$ is a representation on $A$ satisfying \eqref{repKalgebras1}-\eqref{repKalgebras3}.
    \end{proof}
    \begin{ex}
It is known  that $(A,ad)$ is a representation of $A$ called the adjoint representation. Then $(A, [\cdot,\cdot],ad)$ is an $A$-module Malcev algebra.
\end{ex}
\begin{prop}\label{repalt=repmalc}
Let $(A,\c)$ be an alternative algebra. Then
the triplet $(V; [\c,\c]_V,\mathfrak{l}-\mathfrak{r})$ defines an $A$-module Malcev admissible algebra of $(A,[\c,\c])$.
\end{prop}
\begin{proof}
 By Proposition \ref{semidirectproductalt},
$A\ltimes_{\mathfrak{l},\mathfrak{r}} V$ is an alternative algebra. For its associated Malcev
 algebra $(A \oplus V, \overbrace{[\c, \c]})$, we have
\begin{align*}
 \overbrace{[x+a, y+b]}& =(x+a)\ast (y+b) - (y+b)\ast (x+a) \\
&= x \cdot y+ \mathfrak{l}(x)b+ \mathfrak{r}(y)a + a\c_V b- y \cdot x - \mathfrak{l}(y)a - \mathfrak{r}(x)b-b\c_V a  \\
&= [x, y] + (\mathfrak{l}-\mathfrak{r})(x)b - (\mathfrak{l} -\mathfrak{r})(y)a+[a,b]_{V}.
\end{align*}
According to \eqref{semidirectproductMalcev}, we deduce that $(V; [\c,\c]_V,\mathfrak{l}-\mathfrak{r})$ is an $A$-module Malcev admissible algebra of $(A,[\c, \c])$.
\end{proof}

\begin{df}
 Let $(A,[\c,\c])$ be a Malcev algebra and  $(V;  [\c,\c]_V,\r)$  be an $A$-module Malcev algebra. A linear map $T:V\to A$ is said to be a \textbf{$\lambda$-weighted $\mathcal{O}$-operator}
  associated to $(V; [\c,\c]_V,\r)$ if for all $a,b\in V$,
 \begin{align} \label{O-operatormalcev}
    &[T(a),T(b)]=T\big(\rho(T(a))b-\rho(T(b))a+\lambda[a,b]_V\big).
 \end{align}
\end{df}
Obviously, a $\lambda$-weighted $\mathcal{O}$-operator associated to $(A, [\c,\c],ad)$ is just  a $\lambda$-weighted Rota-Baxter operator on $A$.
A $\lambda$-weighted $\mathcal{O}$-operator can be viewed as the relative version of a Rota-Baxter operator in the sense
that the domain and range of an $\mathcal{O}$-operator might be different.
\begin{ex}\label{exam-wrb}
\begin{enumerate}[label*={\rm (\roman*)}]
\item A Rota-Baxter operator on $A$ is simply a $0$-weighted $\mathcal{O}$-operator.

\item The identity map $id : A \rightarrow A$ is a $(-1)$-weighted $\mathcal{O}$-operator.

\item If $f : A\rightarrow A$ is a Malcev algebra homomorphism and $f^2 = f$ (idempotent condition), then $f$ is a $(-1)$-weighted $\mathcal{O}$-operator.

\item If $T$ is a $\lambda$-weighted $\mathcal{O}$-operator, then for any $\nu \in \mathbb{K}$, the map $\nu T$ is a $(\nu \lambda)$-weighted $\mathcal{O}$-operator.

\item If $T$ is a $\lambda$-weighted $\mathcal{O}$-operator, then $- \lambda id
- T$ is a $\lambda$-weighted $\mathcal{O}$-operator.
\end{enumerate}
\end{ex}

In the following, we characterize $\lambda$-weighted $\mathcal{O}$-operators in terms of their graph.
\begin{prop}
Let   $(V; [\c,\c]_V,\r)$  be an $A$-module Malcev algebra. Then
 a linear map $T:V\rightarrow A$ is  a $\lambda$-weighted $\mathcal{O}$-operator  associated to $(V,[\c,\c]_V,\r)$ if and only if the graph $$\mathrm{Gr}(T)=\{T(a)+a|~a\in V\}$$ of the map $T$
is a  subalgebra of the  semi-direct product $A\ltimes_{\r} V$.
\end{prop}
\begin{proof}
 Let $T:V\rightarrow A$ be a linear map. For all $a,b\in V$, we have
 \begin{align*}
    [T(a)+a, T(b)+b]_{\rho}=&[T(a),T(b)]+ \rho(T(a))b-\rho(T(b))a+\lambda[a,b]_V,
 \end{align*}
 which implies that the graph $\mathrm{Gr}(T)=\{T(a)+a|~a\in V\}$ is a  subalgebra of  the Malcev algebra $A\ltimes V$ if and only if
 $T$ satisfies
 \begin{align*}
    &[T(a),T(b)]=T(\rho(T(a))b-\rho(T(b))a+\lambda[a,b]_V,
 \end{align*}
 which means that $T$ is a  $\lambda$-weighted $\mathcal{O}$-operator.
\end{proof}
As a consequence of the above proposition, we get the following.
\begin{cor}\label{cor:induced-malc}
Let $T : V \rightarrow A$ be a $\lambda$-weighted $\mathcal{O}$-operator. Since $\mathrm{Gr} (T)$ is isomorphism to $V$ as a vector space, we get that $V$ inherits a new Malcev algebra structure with the bracket
\begin{align*}
[a,b]_T : = \rho (T(a))b - \rho (T(b))a + \lambda [a,b]_V, ~ \text{ for } a, b \in V.
\end{align*}
In other words, $(V, [\c,\c]_T)$ is a Malcev algebra, denoted by $V_T$ (called the induced Malcev algebra). Moreover, $T: V_T \rightarrow A$ is a homomorphism of Malcev algebras.
\end{cor}
Let $T, T': (A,[\cdot,\cdot])\to (V,[\cdot,\cdot]_V)$ be two $\lambda$-weighted $\mathcal{O}$-operators. A $\textbf{homomorphism}$ from $T$ to $T'$ consists of Malcev algebra homomorphisms $\phi: A\to A$ and   $\psi: V\to V$ such that
\begin{eqnarray}
 \label{condition-1}\phi\circ T&=&T'\circ\psi,\\
  \label{condition-2}\psi(\rho(x)a)&=&\rho(\phi(x))(\psi(a)),\quad \forall x\in A, a\in V.
\end{eqnarray}
In particular, if both $\phi$ and $\psi$ are invertible, $(\phi, \psi)$ is called an \textbf{ isomorphism} from $T$ to $T'$.
\begin{prop}
Let $(\phi, \psi)$ be a homomorphism of $\lambda$-weighted $\mathcal{O}$-operators from $T$ to $T'$. Then $\psi: V \rightarrow V$ is a homomorphism of induced Malcev algebras from $(V, [\c,\c]_T)$ to $(V, [\c,\c]_{T'})$.
\end{prop}

\begin{proof}
For any $a, b \in V$, we have
\begin{align*}
\psi ([a,b]_T) =~& \psi \big( \rho (T(a))b - \rho (T(b))a + \lambda [a,b]_V  \big) \\
=~& \rho (\phi(T (a)))(\psi (b)) - \rho (\phi(T (b)))(\psi (a)) + \lambda [\psi (a), \psi (b)]_V \\
=~& \rho(T' (\psi (a)))(\psi (b))  - \rho(T'(\psi (b)))(\psi (a)) + \lambda [\psi (a), \psi (b)]_V = [ \psi (a), \psi (b) ]_{T'}.
\end{align*}
This shows that $\psi : (V, [\c,\c]_T) \rightarrow (V, [\c,\c]_{T'})$ is a homomorphism of Malcev algebras.
\end{proof}
In the sequel, we characterize $\lambda$-weighted $\mathcal{O}$-operators associated to $(V; [\c,\c]_V,\r)$ in terms of the Nijenhuis operators.  Recall  that a Nijenhuis operator on a Malcev algebra $(A,[\c,\c])$ is a linear map $N:A\rightarrow A$  satisfying, for all  $x,y\in A$,   $$[N(x),N(y)]=N\big([N(x),y]-[N(y),x]-N([x,y])\big).$$
\begin{prop}
Let  $(V; [\c,\c]_V,\r)$  be an $A$-module Malcev algebra. Then
 a linear map $T:V\rightarrow A$ is  a $\lambda$-weighted $\mathcal{O}$-operator  associated to $(V; [\c,\c]_V,\r)$ if and only if
$$N_T=\begin{bmatrix}
   id & T \\
    0  & 0
\end{bmatrix}:A\oplus V\rightarrow A\oplus V $$
is a Nijenhuis operator on the semidirect product Malcev algebra $A\ltimes V$.
\end{prop}

\begin{proof}
For all $x,y\in A,$ $a,b\in V$, on the one hand, we have
\begin{align*}
& [N_T(x+a),N_T(y+b)]_{\rho}= [x+T(a),y+T(b)]_{\rho}=[x,y]+[x,T(b)]-[T(a), y] + [T(a),T(b)].
\end{align*}
On the other hand, since $\overline{T}^2=\overline{T},$ we have
\begin{align*}
&
N_T\big([N_T(x+a),y+b]_{\rho}-[N_T(y+b),x+a]_{\rho}-N_T([x+a,y+b]_{\rho})\big)\\
&=N_T\big([N_T(x+a),y+b]_{\rho}-[N_T(y+b),x+a]_{\rho}\big)-N_T([x+a,y+b]_{\rho})\\
& =N_T \big([x +(a),y+b]-[y+T(b),x+a]\big)-N_T\big([x,y]+\r(x)b-\r(y)a+\lambda[a,b]_V\big)\\
& =[x,y]+[x,T(b)]-[T(a), y]+ T\big(\rho(T(a))b-\rho(T(b))a +\lambda[a,b]_V\big),
\end{align*}
which implies that $N_T$ is a Nijenhuis operator on the semi-direct product Malcev algebra $A\ltimes V$ if and only if \eqref{O-operatormalcev} is satisfied.
\end{proof}

\subsection{Definition and Constructions of post-Malcev algebras}
In this section, we introduce the notion of post-Malcev algebras. We show that post-Malcev algebras arise naturally from a $\lambda$-weighted $\mathcal{O}$-operators. Therefore, post-Malcev algebras can be viewed as the underlying
algebraic structures of $\lambda$-weighted $\mathcal{O}$-operators  on Malcev algebras. Finally, we study some properties of post-Malcev algebras.

\begin{df}\label{def:postmalcev}
A $\textbf{post-Malcev algebra}$ $(A,[\c,\c],\rhd)$ is a Malcev algebra $(A,[\c,\c])$
together with a bilinear map $\rhd : A\otimes A \to A$ such that  for all $x,y,z \in A$,
and $\{x, y\} = x\rhd y- y\rhd x + [x, y]$,
\begin{align}
\label{postmalcevalgebras1}
&\{x, z\}\rhd [y, t] = x\rhd[z\rhd y, t]- [z\rhd(x\rhd y), t] - [x\rhd(z\rhd t), y] + z\rhd [x\rhd t, y],
\\
\label{postmalcevalgebras2}
&[x\rhd z, y\rhd t]= [\{x, y\}\rhd z, t] -  x\rhd[y\rhd z, t] +  y\rhd(x\rhd[z, t]) + [y\rhd(x\rhd t), z], \\
\label{postmalcevalgebras3}
&[x\rhd z, [y, t]] = [[x\rhd y, z], t] - x\rhd [[y, z], t] - [x\rhd[z, t], y] - [[x\rhd t, y], z],
\\
\label{postmalcevalgebras4}
&\{y,z\}\rhd(x\rhd t)=\{\{x,y\},z\}\rhd t+y\rhd(\{x,z\}\rhd t)+x\rhd(y\rhd( z\rhd t))-z\rhd(x\rhd( y\rhd t)).
\end{align}
\end{df}
\begin{ex}~~
\begin{enumerate}
\item A pre-Malcev algebra is a post-Malcev algebra with an abelian Malcev algebra $(A, [\cdot, \cdot] = 0, \rhd)$.
\textup{(}See   \textup{\cite{Fattoum,Madariaga}} for more details.\textup{)}
  \item Post-Malcev algebras generalize post-Lie algebras.
  \item  If $(A,[\c,\c])$ is a  Malcev algebra, then $(A,[\c,\c],\rhd)$ is a post-Malcev algebra, where $x\rhd y=[y,x]$ for all $x,y\in A$.
\end{enumerate}
\end{ex}

Let $(A, [\cdot, \cdot], \rhd)$ and $(A', [\cdot, \cdot]' , \rhd')$ be two post-Malcev algebras. A
homomorphism of post-Malcev algebras is a linear map $f : A \to A'$ such that
$ f([x, y]) = [f(x), f(y)]'$ and $f(x \rhd y) = f(x) \rhd' f(y).$
\begin{prop}\label{postmalcev==>malcev}
Let $(A,[\c,\c],\rhd)$ be a post-Malcev algebra. Then the bracket
\begin{equation}\label{malcevbracket}
      \{x,y\}=x\rhd y- y\rhd x+ [x,y]
\end{equation}
defines a Malcev algebra structure on $A$. We denote this algebra by $A^C$ and we call it the sub-adjacent  Malcev algebra of $A$.
\end{prop}
\begin{proof}
The skew symmetry  is obvious. For all $x, y, z, t \in A$, we have
\begin{align*}
&
\{\{x, z\}, \{y, t\}\} = \{x, z\}\rhd\{y, t\}-\{y, t\}\rhd\{x, z\}+[\{x, z\},\{y, t\}]\\
&=\{x,z\}\rhd(y\rhd t)-\{x,z\}\rhd(t\rhd y)+\{x, z\}\rhd[y, t] - \{y, t\}\rhd(x\rhd z)\\
& \quad+ \{y, t\}\rhd(z\rhd x)-\{y, t\}\rhd[x, z]+[x\rhd z, y\rhd t]- [x\rhd z, t\rhd y]\\
&\quad+ [x\rhd z, [y, t]] - [z\rhd x, y\rhd t]+ [z\rhd x, t\rhd y] - [z\rhd x, [y, t]]\\
&\quad+[[x, z], y\rhd t]-[[x, z], t\rhd y]+ [[x, z], [y, t]],
\\[0,1cm]
 &\{\{\{x, y\}, z\}, t\}= \{\{x, y\}, z\}\rhd t- t\rhd\{\{x, y\}, z\}+ [\{\{x, y\}, z\},t]\\
 &=\{\{x, y\}, z\}\rhd t- t\rhd(\{x, y\}\rhd z)+ t\rhd(z\rhd(x\rhd y))- t\rhd(z\rhd(y\rhd x))\\
 &\quad+ t\rhd(z\rhd[x, y])- t\rhd[x\rhd y, z]+  t\rhd[y\rhd x, z]-t\rhd[[x, y], z]\\
 &\quad+ [\{x, y\}\rhd z, t]- [z\rhd(x\rhd y), t]+ [z\rhd(y\rhd x), t]- [z\rhd[x, y], t]\\
 &\quad + [[x\rhd y,z], t] - [[y\rhd x,z], t]  + [[[x, y], z], t],
\\[0,1cm]
  &\{\{\{y, z\}, t\}, x\}= \{\{y, z\}, t\}\rhd x-  x\rhd\{\{y, z\}, t\}+[\{\{y, z\}, t\},x]\\
  &=\{\{y, z\}, t\}\rhd x- x\rhd(\{y, z\}\rhd t) + x\rhd(t\rhd(y\rhd z)) - x\rhd(t\rhd(z\rhd y))\\
  &\quad+ x\rhd(t\rhd[y, z])- x\rhd[y\rhd z, t]+  x\rhd[z\rhd y, t]-x\rhd[[y, z], t]\\
  &\quad+ [\{y, z\}\rhd t, x]- [t\rhd(y\rhd z), x]+ [t\rhd(z\rhd y), x]-[t\rhd [y,z],x]\\
  & \quad+ [[y\rhd z,t], x]- [[z\rhd y, t], x] + [[[y, z], t], x],
\\[0,1cm]
  &\{\{\{z, t\}, x\}, y\}=\{\{z, t\}, x\}\rhd y-y\rhd \{\{z, t\}, x\}+ [\{\{z, t\}, x\},y]\\
   &=\{\{z, t\}, x\}\rhd y- y\rhd(\{z, t\}\rhd x)+ y\rhd(x\rhd(z\rhd t))- y\rhd(x\rhd(t\rhd z))\\
   &\quad+ y\rhd(x\rhd[z, t])- y\rhd[z\rhd t, x] +  y\rhd[t\rhd z, x]-y\rhd[[z, t], x]\\
   &\quad+ [\{z, t\}\rhd x, y]- [x\rhd(z\rhd t), y]+ [x\rhd(t\rhd z), y]- [x\rhd[z, t], y]\\
   &\quad+ [[z\rhd t,x], y]- [[t\rhd z, x], y] + [[[z, t], x], y],
\\[0,1cm]
  &\{\{\{t, x\}, y\}, z\}=\{\{t, x\}, y\}\rhd  z-z\rhd \{\{t, x\}, y\} +[\{\{t, x\}, y\},z]\\
   &=\{\{t, x\}, y\}\rhd  z- z\rhd(\{t, x\}\rhd y)+ z\rhd(y\rhd(t\rhd x))- z\rhd(y\rhd(x\rhd t))\\
   &\quad+ z\rhd(y\rhd[t, x])- z\rhd[t\rhd x, y]+  z\rhd[x\rhd t, y]-z\rhd[[t, x], y]\\
   &\quad+ [\{t, x\}\rhd y, z]- [y\rhd(t\rhd x), z]+ [y\rhd(x\rhd t), z]- [y\rhd[t, x], z]\\
   &\quad+ [[t\rhd x, y], z]- [[x\rhd t, y], z]+ [[[t, x], y], z].
\end{align*}
By the identity of Malcev algebra and  \eqref{postmalcevalgebras1}-\eqref{postmalcevalgebras4},
we have
\begin{equation*}
\{\{x, z\}, \{y, t\}\}- \{\{\{x, y\}, z\}, t\} - \{\{\{y, z\}, t\}, x\} - \{\{\{z, t\}x\}, y\} - \{\{\{t, x\}, y\}, z\} = 0.
\qedhere
\end{equation*}
\end{proof}
\begin{rem}
Let $(A,[\c,\c],\rhd)$ be a post-Malcev algebra. If $\rhd$ is commutative, $x\rhd y = y\rhd x$, then the two Malcev brackets
 $[\c,\c]$ and $\{\c,\c\}$ coincide.
\end{rem}
\begin{cor}
If $(A,[\c,\c],\rhd)$ be a post-Malcev algebra, then
$(A,\circ)$ is an admissible Malcev algebra,
with the product $\circ$ defined for all $x,y\in A$ by
  \begin{equation}\label{Admissibilemalcev}
    x\circ y=x\rhd y+\frac{1}{2}[x,y].
  \end{equation}
\end{cor}
\begin{prop}
Let $(A,[\c,\c],\rhd)$ be a post-Malcev algebra. Define
$L_\rhd:A\to A$ by $L_{\rhd}(x)y=x\rhd y$ for any $x,y\in A$.
Then
$(A^{C}; [\c,\c], L_\rhd)$ is an $A$-module Malcev algebra of  $(A^{C},\{\c,\c\})$.
\end{prop}
\begin{proof}
By \eqref{postmalcevalgebras4}, $L_\rhd$ is a representation of $(A^{C},\{\c,\c\})$. Indeed, for $x,y,z, t \in A$,
\begin{align*}
L_\rhd(\{y,z\})L_\rhd(x)&t=\{y,z\}\rhd(x\rhd t)= \{\{x,y\},z\}\rhd t+y\rhd(\{x,z\}\rhd t)+x\rhd(y\rhd( z\rhd t))-z\rhd(x\rhd( y\rhd t))\\
&=L_\rhd (\{\{x,y\},z\}) t+L_\rhd (y)L_\rhd(\{x,z\}) t+L_\rhd (x)L_\rhd(y)L_\rhd(z)t-L_\rhd(z)L_\rhd(x)L_\rhd( y)t.
\end{align*}
To prove \eqref{repKalgebras1}, according to \eqref{postmalcevalgebras1} we compute
\begin{align*}
L_\rhd(\{x, z\})[y, t]& = \{x, z\}\rhd [y, t]= x\rhd[z\rhd y, t]- [z\rhd(x\rhd y), t] - [x\rhd(z\rhd t), y] + z\rhd [x\rhd t, y]\\
&=L_\rhd(x)[L_\rhd(z)y, t]- [L_\rhd(z)L_\rhd(x)y, t] - [L_\rhd(x)L_\rhd(z)t, y]+ L_\rhd(z)[L_\rhd(x)t, y].
 \end{align*}
 Similarly, by \eqref{postmalcevalgebras2} and \eqref{postmalcevalgebras3}, we have
 \begin{align*}
     [L_\rhd(x) z, L_\rhd(y)t]&=[x\rhd z, y\rhd t]= [\{x, y\}\rhd z, t] -  x\rhd[y\rhd z, t] +  y\rhd(x\rhd[z, t]) + [y\rhd(x\rhd t), z]\\
     &=[L_\rhd(\{x, y\}) z, t] -  L_\rhd(x)[L_\rhd(y)z, t] +  L_\rhd(y)L_\rhd(x)[z, t] + [L_\rhd(y)L_\rhd(x) t, z],
    \\[0,2cm]
     [L_\rhd(x) z, [y, t]]& =[x\rhd z, [y, t]] =[[x\rhd y, z], t] - x\rhd [[y, z], t] - [x\rhd[z, t], y] - [[x\rhd t, y], z]\\
     &=[[L_\rhd(x)y, z], t] - L_\rhd(x)[[y, z], t] - [L_\rhd(x)[z, t], y] - [[L_\rhd(x) t, y], z].
 \end{align*}
Therefore $(A^{C}; [\c,\c], L_\rhd)$ is an $A$-module Malcev algebra of  $(A^{C},\{\c,\c\})$.
\end{proof}
\begin{prop}
If $(A,[\c,\c],\rhd)$ be a post-Malcev algebra, then  $(A,-[\c,\c],\blacktriangleright)$
 is  also a post-Malcev algebra, where for all $x,y\in A$,
  \begin{equation}
x\blacktriangleright y=x \rhd y+ [x,y].
\end{equation}
Moreover, $(A,[\c,\c],\rhd)$  and  $(A,-[\c,\c],\blacktriangleright)$ have the same sub-adjacent Malcev algebra $A^C$.
\end{prop}
\begin{proof}
We check only that $(A,-[\cdot,\cdot],\blacktriangleright)$ verifies the first post-Malcev identity. The other identities can be verified similarly.
  In fact, for all $x,y,z, t\in A$,
\begin{align*}
&\{x, z\}\blacktriangleright [y, t] - x\blacktriangleright[z\blacktriangleright y, t]+ [z\blacktriangleright(x\blacktriangleright y), t] + [x\blacktriangleright(z\blacktriangleright t), y] - z\blacktriangleright [x\blacktriangleright t, y]\\
&=\{x, z\}\rhd [y, t] + [\{x, z\}, [y, t]] - x\rhd[z\rhd y, t]- x\rhd [[z, y], t] - [x, [z\rhd y, t]]\\
&\quad - [x, [[z, y], t]]-[z\rhd(x\rhd y), t] - [z\rhd [x, y], t]+ [[z, x\rhd y], t]+ [[z, [x, y]], t]\\
&\quad+  [x\rhd(z\rhd t), y] +[x\rhd[z, t], y] + [[x, z\rhd t], y] + [[x, [z, t]], y]-  z\rhd [x\rhd t, y]\\
&\quad - z\rhd[[x, t], y]- [z, [x\rhd t, y]] - [z, [[x, t], y]] =0.
\qedhere
\end{align*}

\end{proof}
\begin{thm}
If $(A,[\c,\c],\rhd)$ is a post-Malcev algebra, then
$(A\times A,\llbracket \c,\c \rrbracket)$ is  a Malcev algebra, with the double bracket product $\llbracket \c,\c \rrbracket$ on $A\times A$ defined for all $a,b,x,y\in A$ by
  \begin{equation}
  \label{DoubleBarcket}
    \llbracket (a,x),(b,y) \rrbracket =(a\rhd b-b\rhd a+[a,b], \quad
    a\rhd y-b\rhd x +[x,y]).
  \end{equation}
\end{thm}
\begin{proof}
Let $x,y,z, t, a,b,c, d\in A$. It is obvious that $\llbracket (a,x),(b, y) \rrbracket=-\llbracket (b, y),(a, x) \rrbracket.$
 On the other hand,
 \begin{align*}
 &\llbracket\llbracket (a, x), (c, z)\rrbracket, \llbracket(b, y), (d, t)\rrbracket\rrbracket= \\
 & \{\{a, c\}, \{b, d\}\},~~(a\rhd c)\rhd(b\rhd t)- (a\rhd c)\rhd(d\rhd y) -(c\rhd a)\rhd(b\rhd t)\\
 & + (c\rhd a)\rhd(d\rhd y) + [a, c]\rhd(b \rhd t) - [a, c]\rhd(d \rhd y) + \{a, c\}\rhd[y, t]\\
 & -(b\rhd d)\rhd(a\rhd z)+ (b\rhd d)\rhd(c\rhd x) + (d\rhd b)\rhd(a\rhd z) - (d\rhd b)\rhd(c\rhd x)\\
 & - [b, d]\rhd(a \rhd z) + [b, d]\rhd(c \rhd x) - \{b, d\}\rhd[x, z] + [a\rhd z, b\rhd t]\\
 &- [a\rhd z, d\rhd y] -[c\rhd x, b\rhd t] + [c\rhd x, d\rhd y]- [a\rhd z, [y, t]] +[c\rhd x, [y, t]]\\
 &+ [[x, z], b\rhd t]  -[[x, z], d\rhd y] -[[x, z], [y, t]],
 \\[0,1cm]
 &\llbracket\llbracket \llbracket(a, x), (b, y)\rrbracket, (c, z)\rrbracket, (d, t)\rrbracket=\\
 &\{\{\{a, b\}, c\}, d\},~~((a\rhd b)\rhd c)\rhd t -((b \rhd a)\rhd c)\rhd t- (c\rhd(a\rhd b))\rhd t\\
 &+ (c\rhd (b\rhd a)) \rhd t + [a\rhd b, c]\rhd t- [b\rhd a, c]\rhd t+ \{[a, b], c\}\rhd t\\
 &- d\rhd((a\rhd b)\rhd z)+ d\rhd ((b\rhd a)\rhd z) - d\rhd([a, b]\rhd z)+ d\rhd(c\rhd(a\rhd y))\\
 & - d\rhd(c\rhd(b\rhd x))+ d\rhd(c\rhd[x, y])-d\rhd[a\rhd y, z]+ d\rhd [b\rhd x, z]- d\rhd [[x, y], z]\\
 &+ [\{a, b\}\rhd z, t]- [c\rhd(a\rhd y), t]+ [c\rhd(b\rhd x), t]- [c\rhd [x, y], t]\\
 &+ [[a\rhd y, z], t]- [[b\rhd x, z], t]+ [[[x, y], z], t],
\\[0,1cm]
 &\llbracket\llbracket \llbracket(b, y), (c, z)\rrbracket, (d, t)\rrbracket, (a, x)\rrbracket=\\
 &\{\{\{b, c\}, d\}, a\},~~((b\rhd c)\rhd d)\rhd x -((c \rhd b)\rhd d)\rhd x- (d\rhd(b\rhd c))\rhd x\\
 &+ (d\rhd (c\rhd b)) \rhd x + [b\rhd c, d]\rhd x- [c\rhd b, d]\rhd x+ \{[b, c], d\}\rhd x\\
 &- a\rhd((b\rhd c)\rhd t)+ a\rhd ((c\rhd b)\rhd t) - a\rhd([b, c]\rhd t)+ a\rhd(d\rhd(b\rhd z))\\
 & - a\rhd(d\rhd(c\rhd y))+ a\rhd(d\rhd[y, z])-a\rhd[b\rhd z, t]+ a\rhd [c\rhd y, t]- a\rhd [[y, z], t]\\
 &+ [\{b, c\}\rhd t, x]- [d\rhd(b\rhd z), x]+ [d\rhd(c\rhd y), x]- [d\rhd [y, z], x]\\
 &+ [[b\rhd z, t], x]- [[c\rhd y, t], x]+ [[[y, z], t], x],
\\[0,1cm]
 &\llbracket\llbracket \llbracket(c, z), (d, t)\rrbracket, (a, x)\rrbracket, (b, y)\rrbracket=\\
 &\{\{\{c, d\}, a\}, b\},~~((c\rhd d)\rhd a)\rhd y -((d \rhd c)\rhd a)\rhd y- (a\rhd(c\rhd d))\rhd y\\
 &+ (a\rhd (d\rhd c)) \rhd y + [c\rhd d, a]\rhd y- [d\rhd c, a]\rhd y+ \{[c, d], a\}\rhd y\\
 &- b\rhd((c\rhd d)\rhd x)+ b\rhd ((d\rhd c)\rhd x) - b\rhd([c, d]\rhd x)+ b\rhd(a\rhd(c\rhd t))\\
 & - b\rhd(a\rhd(d\rhd z))+ b\rhd(a\rhd[z, t])- b\rhd[c\rhd t, x]+ b\rhd [d\rhd z, x]- b\rhd [[z, t], x]\\
 &+ [\{c, d\}\rhd x, y]- [a\rhd(c\rhd t), y]+ [a\rhd(d\rhd z), y]- [a\rhd [z, t], y]\\
 &+ [[c\rhd t, x], y]- [[d\rhd z, x], y]+ [[[z, t], x], y],
\\[0,1cm]
 &\llbracket\llbracket \llbracket(d, t), (a, x)\rrbracket, (b, y)\rrbracket, (c, z)\rrbracket=\\
 &\{\{\{d, a\}, b\}, c\},~~((d\rhd a)\rhd b)\rhd z -((a \rhd d)\rhd b)\rhd z- (b\rhd(d\rhd a))\rhd z\\
 &+ (b\rhd (a\rhd d)) \rhd z + [d\rhd a, b]\rhd z- [a\rhd d, b]\rhd z+ \{[d, a], b\}\rhd z\\
 &- c\rhd((d\rhd a)\rhd y)+ c\rhd ((a\rhd d)\rhd y) - c\rhd([d, a]\rhd y)+ c\rhd(b\rhd(d\rhd x))\\
 & - c\rhd(b\rhd(a\rhd t))+ c\rhd(b\rhd[t, x])-c\rhd[d\rhd x, y]+ c\rhd [a\rhd t, y]- c\rhd [[t, x], y]\\
 &+ [\{d, a\}\rhd y, z]- [y\rhd(d\rhd x), z]+ [y\rhd(a\rhd t), z]- [y\rhd [t, x], z]\\
 &+ [[d\rhd x, y], z]- [[a\rhd t, y], z]+ [[[t, x], y], z].
 \end{align*}
 Hence, using \eqref{malcevbracket} of Proposition \ref{postmalcev==>malcev} and  Definition \ref{def:postmalcev}, we have
  \begin{align*}
&\llbracket\llbracket (a, x), (c, z)\rrbracket, \llbracket(b, y), (d, t)\rrbracket\rrbracket- \llbracket\llbracket \llbracket(a, x), (b, y)\rrbracket, (c, z)\rrbracket, (d, t)\rrbracket\\
& -\llbracket\llbracket \llbracket(b, y), (c, z)\rrbracket, (d, t)\rrbracket, (a, x)\rrbracket -\llbracket\llbracket \llbracket(c, z), (d, t)\rrbracket, (a, x)\rrbracket, (b, y)\rrbracket\\
 &- \llbracket\llbracket \llbracket(d, t), (a, x)\rrbracket, (b, y)\rrbracket, (c, z)\rrbracket = (0, 0).
 \qedhere
 \end{align*}
 \end{proof}
 The following results illustrate that a $\lambda$-weighted $\mathcal{O}$-operator induces a post-Malcev algebra structure.
 \begin{thm}
  Let $(A, [\c,\c]_A)$ be  a Malcev algebra and
  $(V; [\c,\c]_V, \r)$  an $A$-module Malcev algebra.
Let  $T: V\to A$  be a  $\lambda$-weighted $\mathcal{O}$-operator associated to
$(V; [\c,\c]_V, \r)$.
\begin{enumerate}[label*={\rm (\roman*)}]
\item
\label{it:depostmal1}
Define two new bilinear operations $[\c,\c],\ \rhd: V\times V\to V$ as follows, for all $a, b\in V$,
\begin{eqnarray}\label{malcev==>Postmalcev}
[a,b]=\lambda[a,b]_V,\quad a\vartriangleright b=\r(T(a))b.
\end{eqnarray}
Then $(V, [\c,\c], \vartriangleright)$ is a post-Malcev algebra.
\item \label{it:depostmal2}
$T$ is a Malcev algebra homomorphism from the sub-adjacent Malcev algebra  $(V, \{\c,\c\})$ given in Proposition \ref{postmalcev==>malcev} to $(A, [\c,\c]_A)$.
\end{enumerate}
\label{thm:depostmal}
\end{thm}
\begin{proof}
\ref{it:depostmal1}  We use \eqref{repKalgebras1}-\eqref{repKalgebras3} of representation of Malcev algebras on $\mathbb{K}$-algebra.
\begin{align*}
&\{a, c\}\rhd [b, d] - a\rhd[c\rhd b, d]+ [c\rhd(a\rhd b), d] + [a\rhd(c\rhd d), b] - c\rhd [a\rhd d, b]\\
& \quad =\lambda\Big(\r([T(a), T(b)]_A)[b, d]_V - \r(T(a))[\r(T(c))b, d]_V + [\r(T(c))\r(T(a))b, d]_V\\
&\quad \quad + [\r(T(a))\r(T(c))d, b]_V- \r(T(c))[\r(T(a))d, b]_V\Big)=0,
\\[0,1cm]
&[a\rhd c, b\rhd d]-  [\{a, b\}\rhd c, d] +  a\rhd[b\rhd c, d] -  b\rhd(a\rhd[c, d]) - [b\rhd(a\rhd d), c]\\
& \quad =\lambda \Big([\r(T(a))c, \r(T(b))d]_V - [\r([T(a), T(b)]_A )c, d]_V+ \r(T(a))[\r(T(b))c, d]_V\\
&\quad \quad - \r(T(b))\r(T(a))[c, d]_V- [\r(T(b))\r(T(a))d, c]_V\Big)=0,
\\[0.2cm]
&[a\rhd c, [b, d]] - [[a\rhd b, c], d] - a\rhd [[b, c], d] + [a\rhd[c, d], b] +[[a\rhd d, b], c]\\
&\quad =\lambda^{2}\Big([\r(T(a))c, [b, d]_V]_V- [[\r(T(a))b, c]_V, d]_V- \r(T(a))[[b, c]_V, d]_V \\
&\quad \quad + [\r(T(a))[c, d]_V, b]_V - [[\r(T(a))d, b]_V, c]_V\Big)= 0.
\end{align*}
Using the condition \eqref{representation} of Definition \ref{df:repmalcevalgebra}, we check
\begin{align*}
&(b \rhd c) \rhd (a \rhd d) - (c \rhd b) \rhd (a \rhd d) + ((a \rhd b) \rhd c) \rhd d - ((b \rhd a) \rhd c) \rhd d\\
 &\quad- (c \rhd (a \rhd b)) \rhd d + (c \rhd (b \rhd a)) \rhd d+ b \rhd ((a \rhd c) \rhd d) - b \rhd ((c \rhd a) \rhd d)\\
  &\quad- a \rhd (b \rhd (c \rhd d)) + c \rhd (a \rhd (b \rhd d)) -[a, c]\rhd(b\rhd d)+ [a\rhd b, c]\rhd d\\
  &\quad- [b\rhd a, c]\rhd d+ \{[a, b], c\}\rhd d+ a\rhd([b, c]\rhd d)\\
&= \rho(T(\rho(T(b))c))\rho(T(a))d- \rho(T(\rho(T(c))b))\rho(T(a))d+ \rho(T(\rho(T(\rho(T(a))b))c))d\\
&\quad -  \rho(T(\rho(T(\rho(T(b))a))c))d - \rho(T(\rho(T(c)))\rho(T(a))b)d + \rho(T(\rho(T(c)))\rho(T(b))a)d\\
&\quad + \rho(T(b))\rho(T(\rho(T(a))c))d-  \rho(T(b))\rho(T(\rho(T(c))a))d- \rho(T(a))\rho(T(b))\rho(T(c))d\\
&\quad+ \rho(T(\rho(T(c))b))\rho(T(a))d-\r(T(\lambda[a, c]_V))\r(T(b))d+ \r(T(\lambda[\r(T(a))b, c]_V))d\\
&\quad - \r(T(\lambda[\r(T(b))a, c]_V))d+ \r(T(\{\lambda[a, b]_V, c\}))d + \r(T(a))\r(T(\lambda[b, c]_V))d \\
 &=  \rho([T(b), T(c)]_A)\rho(T(a))d +  \rho([T(\r(T(a))b), T(c)]_A)d- \rho([T(\r(T(b))a), T(c)]_A)d\\
 &\quad- \rho(T(a))\rho(T(b))\rho(T(c))d- \rho(T(b))\rho([T(a), T(c)]_A)d + \rho(T(c))\rho(T(a))\rho(T(b))d\\
  & =\rho([T(b), T(c)]_A)\rho(T(a))d + \rho([[T(a), T(b)]_A, T(c)]_A)d - \rho(T(a))\rho(T(b))\rho(T(c))d\\
  &\quad - \rho(T(b))\rho([T(a), T(c)]_A)d + \rho(T(c))\rho(T(a))\rho(T(b))d
  =0.
\end{align*}
\ref{it:depostmal2} The Malcev bracket $\{\c,\c\}$  is defined for all $a,b\in V$ by
$$\{a,b\}=a\rhd b-b\rhd a+[a,b]=\r(T(a))b-\r(T(b))a+\lambda [a,b]_{V}.$$
Then  the  sub-adjacent Malcev algebra of the above post-Malcev algebra $(V, [\c,\c], \vartriangleright)$
is exactly the Malcev algebra $(V, [\c,\c]_T )$ given in Corollary \ref{cor:induced-malc} Then the result follows.
\end{proof}
\begin{prop}\label{pro:mor-o-op-post-Malcev}
Let $T,T':(V,[\cdot,\cdot]_V)\to (A,[\cdot,\cdot])$ be two $\lambda$-weighted $\mathcal{O}$-operators with respect to an $A$-module Malcev algebra $(V; [\cdot, \cdot]_V, \r)$. Let $(V, \{\c,\c\}, \vartriangleright)$ and $(V, \{\c,\c\}', \vartriangleright')$ be the  post-Malcev algebras given in Theorem \ref{thm:depostmal} and $(\phi, \psi)$ be a homomorphism from $T'$ to $T$. Then $\psi$ is a homomorphism from the post-Malcev algebra $(V, \{\c,\c\}, \vartriangleright)$ to the post-Malcev algebra $(V, \{\c,\c\}', \vartriangleright')$.
\end{prop}
\begin{proof}
For all $a,b\in V,$ by \eqref{condition-1},\eqref{condition-2} and \eqref{malcev==>Postmalcev}, we have
\begin{eqnarray*}
 \psi(\{a,b\})&=&\psi(\lambda[a, b]_{V})=\lambda[\psi(a),\psi(b)]_{V}
=\{\psi(a),\psi(b)\}',\\
\psi(a\vartriangleright b)&=&\psi(\rho(T(a))b)=\rho(\phi(T(a)))(\psi(b))=\rho(T'(\psi(a)))(\psi(b))=
\psi(a)\vartriangleright'  \psi(b),
\end{eqnarray*}
which implies that $\psi$ is a homomorphism between the post-Malcev algebras in Theorem \ref{thm:depostmal}.
\end{proof}
Given a Malcev algebra, the following result gives a necessary and sufficient condition to have a compatible
post-Malcev algebra structure.
\begin{prop}
   Let $(A, [\c,\c])$ be a Malcev algebra. Then there exists  a compatible post-Malcev algebra structure on $A$ if and only if there exists an $A$-module Malcev algebra $(V; [\c,\c]_V, \rho)$ and an invertible $1$-weighted $\mathcal{O}$-operator $T:V\to A$.
\end{prop}
\begin{proof}
Let $(A,[\c,\c],\rhd)$ be a post-Malcev algebra and $(A,[\c,\c])$ be the associated Malcev algebra.  Then the identity map $id: A \to A$ is an  invertible $1$-weighted $\mathcal{O}$-operator on  $(A,[\c,\c])$  associated to
$(A, [\c,\c],ad)$.

Conversely, suppose that there exists an invertible $1$-weighted $\mathcal{O}$-operator  of $(A,[\c,\c])$  associated to  an $A$-module Malcev algebra
 $(V; [\c,\c]_V, \r)$ . Then, using Theorem \ref{thm:depostmal},  there is a post-Malcev algebra structure on $T(V)=A$ given by
\begin{equation*}
  \{x,y\}= \lambda T([T^{-1}(x),T^{-1}(y)]_V),\quad
  x\rhd y=T(\rho(x)T^{-1}(y)).
\end{equation*}
This is compatible post-Malcev algebra structure  on $(A,[\c,\c])$. Indeed,
\begin{align*}
x\rhd y-y\rhd x+ \{x,y\}
&= T(\rho(x)T^{-1}(y) - \rho(y)T^{-1}(x) +  [T^{-1}(x),T^{-1}(y)]_V)\\
&= [TT^{-1}(x),TT^{-1}(y)]=[x,y].
\qedhere
\end{align*}
\end{proof}
An obvious consequence of Theorem \ref{thm:depostmal} is the following construction of a post-Malcev  algebra in terms of $\lambda$-weighted Rota-Baxter operator  on a Malcev algebra.
\begin{cor}\label{PostMalcevByRotaBaxter}
  Let $(A, [\c,\c])$ be a Malcev algebra and the linear map $\mathcal{R}: A\rightarrow A$  is  a $\lambda$-weighted  Rota-Baxter operator. Then, there exists a post-Malcev structure on $A$ given,
  for all $x,y\in A$, by
\begin{eqnarray*}
\{x,y\}=\lambda[x,y],\quad  x\vartriangleright y=[\mathcal{R}(x),y].
\end{eqnarray*}
If in addition, $\mathcal{R}$ is invertible, then there is a compatible post-Malcev algebra structure on $A$ given, for all $x,y \in A$, by
\begin{equation*}
  \{x,y\}= \mathcal{R}([\mathcal{R}^{-1}(x),\mathcal{R}^{-1}(y)],\quad  x\rhd y=\mathcal{R}([x,\mathcal{R}^{-1}(y)]).
\end{equation*}
\end{cor}
\begin{ex}
In this example, we calculate $(-1)$-weighted Rota-Baxter operators on the Malcev algebra
$A$ and we give the corresponding post-Malcev algebras.
Let $A$ be the simple Malcev algebra over the field of complex numbers $\mathbb{C}$ {\rm \cite[Example 3]{goncharov}}. In this case $A$ has a basis  $\{e_1,e_2,e_3,e_4,e_5, e_6, e_7\}$ with the following table of multiplication:
\begin{center}
\begin{tabular}{c|c|c|c|c|c|c|c}
$[\c, \c]$ & $e_1$ & $e_2$ & $e_3$ & $e_4$ & $e_5$ & $e_6$ & $e_7$ \\
\hline
$e_1$ & $0$ & $2e_2$ & $-2e_3$ & $2e_4$ & $-2e_5$ & $2e_6$ & $-2e_7$\\
\hline
$e_2$ & $-2e_2$ & $0$ & $e_1$ & $2e_7$ & $0$ & $-2e_5$ & $0$  \\
\hline
$e_3$ & $2e_3$ & $-e_1$ & $0$ & $0$ & $-2e_6$ & $0$ & $2e_4$ \\
 \hline
$e_4$ & $-2e_4$ & $-2e_7$ & $0$ & $0$ & $e_1$ & $2e_3$ & $0$ \\
  \hline
$e_5$ & $2e_5$ & $0$ & $2e_6$ & $-e_1$ & $0$ & $0$ & $-2e_2$\\
 \hline
$e_6$ & $-2e_6$ & $2e_5$ & $0$ & $-2e_3$ & $0$ & $0$ & $e_1$\\
 \hline
$e_7$ & $2e_7$ & $0$ & $-2e_4$ & $0$ & $2e_2$ & $-e_1$ & $0$
\end{tabular}.
\end{center}
Now, define the linear map $\mathcal{R}:A\to A$ by
\begin{align*}\label{RotaBaxterA}
  &\mathcal{R}(e_1)=\frac{1}{2}e_1 + 2\alpha e_2 + 2\beta e_5 + 2\gamma e_6, \ \ \mathcal{R}(e_2)=0, \ \ \mathcal{R}(e_3)= e_3 - \alpha e_1 + \delta e_5  - 2\beta e_6,\\
  &\mathcal{R}(e_4)=e_4- \beta e_1 - \delta e_2 + \mu e_6, \ \ \mathcal{R}(e_5) = \mathcal{R}(e_6) =0, \ \ \mathcal{R}(e_7) = e_7 -\gamma e_1 + 2 \beta e_2- \mu e_5.
\end{align*}
Then $\mathcal{R}$ is a $(-1)$-weighted Rota-Baxter operator on $A$. Using Corollary \ref{PostMalcevByRotaBaxter}, we can construct a post-Malcev algebra on $A$ given by
\begin{center}
\begin{tabular}{c|c|c|c|c|c|c|c}
$\{\c, \c\}$ & $e_1$ & $e_2$ & $e_3$ & $e_4$ & $e_5$ & $e_6$ & $e_7$ \\
\hline
$e_1$ & $0$ & $2\lambda e_2$ & $-2\lambda e_3$ & $2\lambda e_4$ & $-2\lambda e_5$ & $2\lambda e_6$ & $-2\lambda e_7$\\
\hline
$e_2$ & $-2\lambda e_2$ & $0$ & $\lambda e_1$ & $2\lambda e_7$ & $0$ & $-2\lambda e_5$ & $0$  \\
\hline
$e_3$ & $2\lambda e_3$ & $-\lambda e_1$ & $0$ & $0$ & $-2\lambda e_6$ & $0$ & $2\lambda e_4$ \\
\hline
$e_4$ & $-2\lambda e_4$ & $-2\lambda e_7$ & $0$ & $0$ & $\lambda e_1$ & $2\lambda e_3$ & $0$ \\
\hline
$e_5$ & $2\lambda e_5$ & $0$ & $2\lambda e_6$ & $-\lambda e_1$ & $0$ & $0$ & $-2\lambda e_2$\\
\hline
$e_6$ & $-2\lambda e_6$ & $2\lambda e_5$ & $0$ & $-2\lambda e_3$ & $0$ & $0$ & $\lambda e_1$\\
\hline
$e_7$ & $2\lambda e_7$ & $0$ & $-2\lambda e_4$ & $0$ & $2\lambda e_2$ & $-\lambda e_1$ & $0$
\end{tabular}

\end{center}
\footnotesize{
\begin{center}
\begin{tabular}{c|@{\hspace{1mm}}c@{\hspace{1mm}}|@{\hspace{1mm}}c@{\hspace{1mm}}|@{\hspace{1mm}}c@{\hspace{1mm}}|@{\hspace{1mm}}c@{\hspace{1mm}}|@{\hspace{1mm}}c@{\hspace{1mm}}|@{\hspace{1mm}}c@{\hspace{1mm}}|@{\hspace{1mm}}c@{\hspace{1mm}}}
$\rhd$ & $e_1$ & $e_2$ & $e_3$ & $e_4$ & $e_5$ & $e_6$ & $e_7$ \\[2.5mm]
\hline
\rule{0pt}{5mm}
$e_1$ & $\begin{array}{l}4 \beta e_5 -4\alpha e_2 \\ \quad - 4 \gamma e_6\end{array}$ & $e_2 + 4\gamma e_5$ & $\begin{array}{l}2\alpha e_1 -e_3\\ \quad+4\beta e_6\end{array}$ & $\begin{array}{l}e_4 + 4\alpha e_7\\- 2\beta e_1 - 4\gamma e_3\end{array}$ & $-e_5$ & $e_6 - 4\alpha e_5$ & $\begin{array}{l}2\gamma e_1 -e_7\\ \quad-4\beta e_2\end{array}$\\[2.5mm]
\hline
\rule{0pt}{5mm}
$e_2$ & $0$ & $0$ & $0$ & $0$ & $0$ & $0$ & $0$  \\[2.5mm]
\hline
\rule{0pt}{5mm}
$e_3$ & $\begin{array}{l}2e_3 + 2\delta e_5\\ \quad+4\beta e_6\end{array}$ & $\begin{array}{l}-e_1 - 2\alpha e_2\\\quad-4\beta e_5\end{array}$ & $2\alpha e_3 + 2\delta e_6$ & $\begin{array}{l}4\beta e_3 -2\alpha e_4\\\quad-\delta e_1\end{array}$ & $2\alpha e_5 - 2e_6$ & $-2\alpha e_6$ & $\begin{array}{l}2e_4 + 2\alpha e_7\\-2\delta e_2-2 \beta e_1\end{array}$ \\[2.5mm]
\hline
\rule{0pt}{5mm}
$e_4$ & $\begin{array}{l}2\delta e_2 - 2e_4\\\quad- 2\mu e_6\end{array}$ & $\begin{array}{l} 2\mu e_5-2e_7\\\quad-2\beta e_2\end{array}$ & $2\beta e_3 - \delta e_1$ & $\begin{array}{l}-2\beta e_4 - 2\delta e_7\\\quad-2\mu e_3\end{array}$ & $e_1 + 2\beta e_5$ & $\begin{array}{l}2e_3 + 2\delta e_5\\\quad- 2\beta e_6\end{array}$ & $2\beta e_7 + \mu e_1$ \\[2.5mm]
\hline
\rule{0pt}{5mm}
$e_5$ &$0$ & $0$ & $0$ & $0$ & $0$ & $0$ & $0$ \\[2.5mm]
\hline
\rule{0pt}{5mm}
$e_6$ & $0$ & $0$ & $0$ & $0$ & $0$ & $0$ & $0$ \\[2.5mm]
\hline
\rule{0pt}{5mm}
$e_7$ & $\begin{array}{l}2e_7- 4\beta e_2\\\quad- 2\mu e_5\end{array}$ & $-2\gamma e_2$ & $\begin{array}{l}2\beta e_1 +2\gamma e_3\\ -2e_4 - 2\mu e_6\end{array}$ & $\begin{array}{l}\mu e_1 -2\gamma e_4\\\quad+4\beta e_7\end{array}$ & $2e_2+ 2\gamma e_5$ & $\begin{array}{l}-e_1- 2\gamma e_6\\\quad- 4\beta e_5\end{array}$ & $2\gamma e_7+ 2\mu e_2$
\end{tabular}
\end{center}}
\end{ex}

The following result establishes a close relation between a post-alternative algebra and a post-Malcev algebra.

\begin{thm}
Let $T:V\to A$ be a $\lambda$-weighted  $\mathcal{O}$-operator of alternative algebra $(A,\c)$ with respect to
     $(V,\c_V, \mathfrak{l},\mathfrak{r})$ and $(V,\circ,\prec,\succ)$ be the associated  post-alternative algebra given in Theorem \ref{thm:depostalt}. Then $T$ is a $\lambda$-weighted  $\mathcal{O}$-operator on the Malcev admissible algebra $(A,[\c,\c])$  with respect to an $A$-module Malcev algebra    $(V;[\c,\c]_V, \mathfrak{l} - \mathfrak{r})$.

     Moreover, if
 $(V,\{\c,\c\},\rhd)$ be a post-Malcev algebra associated to the Malcev admissible algebra $(A,[\c,\c])$ on  $(V;[\c,\c]_V,\mathfrak{l}-\mathfrak{r})$.
Then, the products $(\{\c,\c\},\rhd)$ are related with $(\circ,\prec,\succ)$ as follow, for all $a,b\in V$,
   \begin{equation}
\label{postaltr to postmalcev } \{a,b\}= a\circ b- b\circ  a, \quad a \rhd b= a \succ b -b\prec a.
\end{equation}
\end{thm}

\begin{proof} Using the condition of $\lambda$-weighted  $\mathcal{O}$-operator in \eqref{oopaltalg} and  Proposition \ref{repalt=repmalc},
for $a, b\in A$,
\begin{align*}
 & [T(a),T(b)]-T\big(\rho(T(a))b-\rho(T(b))a+\lambda[a,b]_V\big)\\
 & =T(a)\cdot T(b)-T(b)\cdot T(a)-T\big((\mathfrak{l}-\mathfrak{r})(T(a))b-(\mathfrak{l}-\mathfrak{r})(T(b))a+\lambda(a\cdot_V b-b\cdot_V a\big)=0.
\end{align*} Then $T$  is a $\lambda$-weighted  $\mathcal{O}$-operator on the Malcev admissible algebra $(A,[\c,\c])$  with respect to an $A$-module Malcev algebra $(V;[\c,\c]_V, \mathfrak{l} - \mathfrak{r})$.

On the other hand, from \eqref{alt==>postalt} of Theorem \ref{thm:depostalt}  and \eqref{malcev==>Postmalcev} of Theorem \ref{thm:depostmal} that
\begin{align*}
&\{a,b\}=\lambda[a,b]_V=\lambda a\c_V b-\lambda b\c_V a=a\circ b-b\circ a,\\
&a\rhd b=(\mathfrak{l}-\mathfrak{r})(T(a))b=\mathfrak{l}(T(a))b-\mathfrak{r}(T(a))b=a\succ b-b\prec a.
\qedhere \end{align*}
\end{proof}

\begin{cor}
Let $(A,\circ,\prec, \succ)$ be a post-alternative algebra given in Corollary \ref{cor:postaltRB},  $(A,\{\c,\c\},\rhd)$ be a post-Malcev algebra associated to the Malcev algebras $(A,[\c,\c])$  and let $\mathcal{R}$ be a $\lambda$-weighted Rota-Baxter operator of  $(A,\c)$. Then, the operations
\begin{equation}
\label{postaltr=>postmalcev} \{x,y\}= x\circ y- y\circ  x, \quad x \rhd y= x \succ y -y\prec x,
\end{equation}
 define a post-Malcev structure in $A$.
\end{cor}
It is easy to see that   \eqref{malcevbracket} and  \eqref{postaltr=>postmalcev}  fit into the commutative diagram
\begin{equation}\label{di:postmalcev}
\xymatrix{
\begin{array}{c}
\substack{\text{Post-alternative} \\ \text{alg.}}
\end{array}
\ar[rr]^{\quad \quad  x\prec y + x\succ y + x\cdot y\quad }
\ar[dd]^{x\rhd y =x\succ y - y\prec x }_{\{x,y\}=x\circ y -y\circ x} &&
\begin{array}{c}
\substack{\text{alternative}\\ \text{alg.}}
\end{array}
\ar[dd]^{x\star y - y\star x} \\
&&\\
\begin{array}{c}
\substack{\text{post-Malcev} \\ \text{alg.}}
\end{array}
\ar[rr]^{\quad x\rhd y- y\rhd x+[x,y]} &&
\begin{array}{c}
\substack{\text{Malcev} \\ \text{alg.}}
\end{array}}
\end{equation}
When the operation $\cdot$ of the post-alternative algebra and the bracket $[\c,\c]$ of the post-Malcev algebra are both trivial, we obtain the following commutative diagram.
$$
\xymatrix{
\begin{array}{c}
\substack{\text{Pre-alternative} \\ \text{alg.}}
\end{array} \ar[rr]^{\quad x\prec y + x\succ y} \ar[dd]^{x\succ y - y\prec x} && \begin{array}{c}
\substack{\text{Alternative} \\ \text{alg.}}
\end{array}
\ar[dd]^{x\star y - y\star x}\\
&& \\
\begin{array}{c}
\substack{\text{Pre-Malcev} \\ \text{alg.}}
\end{array}
\ar[rr]^{\quad x\rhd y - y\rhd x} &&
\begin{array}{c}
\substack{\text{Malcev} \\ \text{alg.}}
\end{array}
}
$$


\end{document}